\documentclass[a4paper]{amsart}
\usepackage{amssymb,amsmath,amsfonts,amsthm}
\usepackage[all]{xy}
\usepackage{enumerate}
\usepackage{mathrsfs}
\usepackage{graphicx}
\usepackage{url} 
\numberwithin{equation}{section} 
\newtheorem{pro}{Proposition}[section]
\newtheorem{lem}[pro]{Lemma}
\newtheorem{cor}[pro]{Corollary} 
\newtheorem{theo}[pro]{Theorem}

\newtheorem*{theoetoile}{Theorem} 
\newtheorem*{conjetoile}{Conjectures} 
\theoremstyle{definition}
\newtheorem{defi}[pro]{Definition}
\newtheorem{exem}[pro]{Example}

\newtheorem{rem}[pro]{Remark} 

\newcommand{\Q}{\mathbb{Q}}
\newcommand{\Z}{\mathbb{Z}}

\newcommand{\F}{\mathcal{F}}

\newcommand{\point}{\mathfrak{p}} 
\newcommand{\q}{\mathfrak{q}}
\newcommand{\D}{\mathcal{D}}
\newcommand{\m}{\mathfrak{m}}
\newcommand{\ia}{\mathfrak{a}}
\newcommand{\ig}{\mathfrak{g}}
\newcommand{\ib}{\mathfrak{b}}

\newcommand{\ZC}{\mathcal{Z}}
\newcommand{\SD}{\mathcal{SD}}
\newcommand{\SDC}{\mathcal{SDC}}
\newcommand{\Ser}{\mathcal{S}}

\newcommand{\Ker}{{\operatorname{Ker}}}
\newcommand{\coker}{{\operatorname{Coker}}}

\renewcommand{\Im}{{\operatorname{Im}}}

\newcommand{\reg}{{\operatorname{reg}}}
\newcommand{\indeg}{{\operatorname{indeg}}}
\newcommand{\ennd}{{\operatorname{end}}}
\newcommand{\hei}{{\operatorname{ht}}}
\newcommand{\depth}{{\operatorname{depth}}}
\newcommand{\Hom}{{\operatorname{Hom}}}
\newcommand{\Homgr}{{\operatorname{Homgr}}}
\newcommand{\Ext}{{\operatorname{Ext}}}
\newcommand{\codim}{{\operatorname{codim}}}

\newcommand{\Spec}{{\operatorname{Spec}}}

\newcommand{\Sym}{{\operatorname{Sym}}}

\newcommand{\length}{{\operatorname{length}}}

\newcommand{\Tot}{{\operatorname{Tot}}}
\newcommand{\Tor}{{\operatorname{Tor}}}

\newcommand{\ann}{{\operatorname{ann}}}
\newcommand{\hor}{{\operatorname{hor}}}
\newcommand{\ver}{{\operatorname{ver}}}

\def\ff{{\bf f}} 
\def\aa{{\bf a}}
\def\xx{{\bf x}}
\def\TT{{\bf T}}

\title{Cohen-Macaulayness and canonical module of residual intersections}

\author{Marc Chardin}
\address{Institut de Math\'{e}matiques de Jussieu. UPMC,  4 place Jussieu, 75005 Paris, France}
\email{marc.chardin@imj-prg.fr}
\urladdr{http://webusers.imj-prg.fr/~marc.chardin/}

\author{Jos\'{e} Na\'{e}liton}
\address{Departamento de Matem\'{a}tica, CCEN, Campus I--sn--Cidade Universit\'{a}ria, Universidade Federal de Para\'{i}ba, 58051-090 Jo\~{a}o Pessoa, Brazil} 
\email{jnaeliton@yahoo.com.br}

\author{Quang Hoa Tran}
\address{University of Education, Hue University,  34 Le Loi St., Hue City, Vietnam
\& Institut de Math\'{e}matiques de Jussieu. UPMC, 4 place Jussieu, 75005 Paris, France}
\email{quang-hoa.tran@imj-prg.fr}
\urladdr{http://webusers.imj-prg.fr/~quang-hoa.tran/}

\date{\today}

\begin{document}
\maketitle

\begin{abstract} 
We show the Cohen-Macaulayness and describe  the canonical module of residual intersections $J=\ia\colon_R I$ in a Cohen-Macaulay local ring $R$, under sliding depth type hypotheses. For this purpose, we construct and study,  using a recent  article of Hassanzadeh and the second named author \cite{HN16}, a family of complexes that contains important informations on a residual intersection and its canonical module.  We also determine  several invariants of residual intersections as the graded canonical module, the Hilbert series, the Castelnuovo-Mumford regularity and the type. Finally, whenever $I$ is strongly Cohen-Macaulay, we show duality results for residual intersections that are closely connected to results by Eisenbud and Ulrich \cite{EU16}. It establishes some tight relations between the Hilbert series of some symmetric powers of $I/\ia$. We also provide closed formulas for the types and for the Bass numbers of some symmetric powers of $I/\ia.$ 
\end{abstract}

\tableofcontents 
\section{Introduction}

The concept of residual intersection was introduced by Artin and Nagata in \cite{AN72}, as a generalization of linkage; it is more ubiquitous, but also harder to understand. Geometrically, let $X$ and $Y$ be two irreducible closed subschemes of a scheme $Z$ with $\codim_Z(X)\leq \codim_Z(Y)=s$ and $Y\nsubseteq X,$ then $Y$ is called a residual intersection of $X$ if the number of equations needed to define $X\cup Y$ as a subscheme of $Z$ is the smallest possible, i.e. $s$. For a ring $R$ and a finitely generated $R$-module $M$, let $\mu_R(M)$ denotes the minimum number of generators of $M$.

The precise definition of a residual intersection is the following.

\begin{defi}
Let $R$ be a Noetherian ring, $I$ be an ideal of height $g$ and $s\geq g$ be an integer.
\begin{enumerate}
\item An \textit{$s$-residual intersection} of $I$ is a proper ideal $J$ of $R$ such that $\hei(J)\geq s$ and $J=(\ia \colon_R I)$ for some ideal $\ia\subset I$ which is generated by $s$ elements.
\item An \textit{arithmetic $s$-residual intersection }of $I$ is an $s$-residual intersection $J$ of $I$ such that $\mu_{R_\point}((I/\ia)_\point)\leq1$ for all prime ideal $\point$ with $\hei(\point)\leq s$.
 \item A \textit{geometric $s$-residual intersection} of $I$ is an $s$-residual intersection $J$ of $I$ such that $\hei(I+J)\geq s+1.$
\end{enumerate}
\end{defi}  

Notice that an $s$-residual intersection is a direct link if $I$ is unmixed and $s =\hei(I).$ Also any geometric $s$-residual intersection is arithmetic.

\smallskip
The theory of residual intersections is a center of interest since the 80's, after Huneke repaired in \cite{Hu83} an argument of Artin and Nagata  in \cite{AN72}, introducing the notion of strongly Cohen-Macaulay ideal: an ideal such that all its Koszul homlogy is Cohen-Macaulay. The notion of strong Cohen-Macaulayness is stable under even linkage, in particular ideals linked to a complete intersection satisfy this property.

\smallskip
In \cite{Hu83} Huneke showed that if $R$ is a Cohen-Macaulay local ring, $J$ is a  $s$-residual intersection of a  strongly Cohen-Macaulay ideal $I$ of $R$ satisfying $G_s,$ then $R/J$ is Cohen-Macaulay of codimension $s.$ Following \cite{AN72}, one says that $I$ satisfies $G_s$ if the number of generators $\mu_{R_\point}(I_\point)$ is at most $\dim(R_\point)$ for all prime ideals $\point$ with $I\subset \point$ and $\dim(R_\point)\leq s-1$ and that $I$ satisfies $G_\infty$ if $I$ satisfies $G_s$ for all $s.$ Later, Herzog, Vasconcelos, and Villarreal in \cite{HVV85} replaced the assumption strong Cohen-Macaulayness by the weaker sliding depth condition, for geometric residuals, but they also showed that this assumption cannot be weakened any further. On the other hand, Huneke and Ulrich proved in \cite{HU88} that the condition $G_s$ is superfluous for ideals in the linkage class of a complete intersection, and more precisely:
\begin{theoetoile}\cite{HU88}
Let  $R$ be a Gorenstein local ring and $I$ be an ideal of height $g$ that is evenly linked to a strongly Cohen-Macaulay ideal satisfying $G_\infty.$ If  $J=\ia\colon_R I$ is an $s$-residual intersection of $I,$ then $R/J$ is Cohen-Macaulay of codimension  $s$ and the canonical module of $R/J$ is the $(s-g+1)$-th symmetric power  of $I/\ia.$
\end{theoetoile}

Let us notice that, in the proof of this statement, it is important to keep track of the canonical module of the residual along the deformation argument that they are using. 

\smallskip
A natural question is then to know if the $G_s$ assumption is at all needed to assert that residuals of ideals that are strongly Cohen-Macaulay, or satisfy the weaker sliding depth condition, are always Cohen-Macaulay, and to describe the canonical module of the residual. In this direction, Hassanzadeh and the second named author remarked in \cite{HN16} that the following long-standing assertions were, explicitely or implicitly, conjectured:
\begin{conjetoile} \cite{HU88, Ul92,CEU01}
Let $R$ be a Cohen-Macaulay local (or $^\ast\!$local) ring and $I$ is strongly Cohen-Macaulay, or even just satisfy sliding depth. Then, for any $s$-residual intersection $J=(\ia\colon_R I)$ of $I,$ 
\begin{enumerate}
\item [\rm (1)] $R/J$ is Cohen-Macaulay.
\item [\rm (2)] The canonical module of $R/J$ is the $(s-g+1)$-th symmetric power of $I/\ia,$ if $R$ is Gorenstein, with $g=\hei(I)\leq s.$
\item [\rm (3)] $\ia$ is minimally generated by $s$ elements.
\item [\rm (4)] $J$ is unmixed.
\item [\rm (5)] When $R$ is positively graded over a field, the Hilbert series of $R/J$ depends only upon $I$ and the degrees of the generators of $\ia.$
\end{enumerate}
\end{conjetoile}

The first conjecture was shown by Hassanzadeh \cite{Ha12} for arithmetic residual intersections, thus in particular for geometric residual intersections, under the sliding depth condition. In the recent article \cite{HN16}, Hassanzadeh and the second named author proved that the second and fifth conjectures hold for the arithmetic residual intersections of strongly Cohen-Macaulay ideals and that  the third and fourth conjectures are true if $\depth(R/I)\geq \dim(R)-s$ and $I$ satisfies the sliding depth condition.  

\smallskip
In this text we will complete the picture, by showing that the first and fifth conjectures hold whenever  $I$ satisfies $\SD_1$ and that the second conjecture is true if $I$ satisfies $\SD_2$ -- recall that an ideal $I=(\ff)=(f_1,\ldots,f_r)$ of height $g$ in a Noetherian local ring $R$ of dimension $d$ satisfies $\SD_k\ (k\geq 0)$ if $\depth(H_i(\ff;R))\geq \min\{d-g,d-r+i+k\}$ for all $i\geq 0;$  note that $\SD_0$ is the sliding depth condition and $\SD_\infty,$ that is $\SD_k$ for all $k\geq 0,$ is strong Cohen-Macaulayness.

\smallskip
In particular all items in the conjecture holds for strongly Cohen-Macaulay ideals. The following puts together part of these results: 

\begin{theoetoile}[Theorems~\ref{Theorem4.5}, \ref{Theorem4.8} and~\ref{Theorem6.2}]
Let $(R,\m)$ be a Cohen-Macaulay local ring with canonical module $\omega$. 
Assume that $J=(\ia\colon_R I)$ is an $s$-residual intersection of $I$ with $\ia\subset I$ and $\hei(I)=g\leq s=\mu_R (\ia )$. Then
\begin{enumerate}
\item [\rm (i)]  $R/J$ is Cohen-Macaulay of codimension $s$ if $I$ satisfies $\SD_1$.
\end{enumerate}
If furthermore $\Tor_1^R(R/I,\omega)=0$, then 
\begin{enumerate}
\item [\rm (ii)]  $\omega_{R/J}\simeq \Sym_R^{s-g+1}(I/\ia)\otimes_R \omega$, provided $I$ satisfies $\SD_2,$
\item [\rm (iii)]  $\omega_{\Sym_R^k(I/\ia)}\simeq \Sym_R^{s-g+1-k}(I/\ia)\otimes_R \omega$ for $1\leq k\leq s-g,$ provided $I$ is strongly Cohen-Macaulay.
\end{enumerate}
\end{theoetoile} 

Notice that $\Tor_1^R(R/I,\omega)=0$ if  $R$ is Gorenstein or $I$ has finite projective dimension.

\smallskip
A key ingredient of our proofs is a duality result between some of the first symmetric powers of $I/\ia$ together with a description of the canonical module of the residual as in items (ii) and (iii) above. This could be compared to recent results of Eisenbud and Ulrich that obtained similar dualities under slightly different hypotheses in \cite{EU16}. In their work, conditions on the local number of generators are needed and depth conditions are asked for some of the first powers of the ideal $I$, along the lines of \cite{Ul94}, and the duality occurs between powers $I^t/\ia I^{t-1}$ in place of symmetric powers $\Sym^t (I/\ia)$. Although their results and ours coincide in an important range of situations, like for geometric residuals of strongly Cohen-Macaulay ideals satisfying $G_s$, the domains of validity are quite distinct. We prove the following.

\begin{theoetoile}[Theorem~\ref{Theorem6.7}]
Let $(R,\m)$ be a Gorenstein local ring  and let  $\ia\subset I$ be two ideals of $R,$ with $\hei(I)=g.$ Suppose that $J=(\ia\colon_R I)$ is an $s$-residual intersection of $I.$ If $I$ is strongly Cohen-Macaulay, then $\omega_{R/J}\simeq \Sym_{R/J}^{s-g+1}(I/\ia)$ and for all $0\leq k\leq s-g+1$
\begin{enumerate}
\item[\em (i)] the $R/J$-module $\Sym_{R/J}^{k}(I/\ia)$ is faithful and Cohen-Macaulay,
\item[\em (ii)] the multiplication
\begin{displaymath}
\Sym_{R/J}^k(I/\ia)\otimes_{R/J} \Sym_{R/J}^{s-g+1-k}(I/\ia)\longrightarrow \Sym_{R/J}^{s-g+1}(I/\ia)
\end{displaymath}
is a perfect pairing, 
\item[\em (iii)] setting $A:=\Sym_{R/J}(I/\ia)$, the graded $R/J$-algebra
$$\overline{A}:=A/A_{>s-g+1}=\bigoplus_{i=0}^{s-g+1}\Sym_{R/J}^i(I/\ia) 
$$
is Gorenstein.
\end{enumerate}
\end{theoetoile}

\smallskip
The paper is organized as follows.

\smallskip
In Section~\ref{Section2}, we collect the notations and general facts about Koszul complexes. We prove duality results for Koszul cycles in Propositions~\ref{Proposition2.2} and~\ref{Proposition2.4}. We also describe the structure of the homology modules of the approximation complexes in Propositions~\ref{Proposition2.5} and~\ref{Proposition2.6}.

\smallskip
In Section~\ref{Section3}, we construct a family of residual approximation complex,  all of same finite size,  $\{_k^M\!\ZC_\bullet^+\}_{k\in \Z}$. This family is a generalization of the family $\{_k\ZC_\bullet^+\}_{k\in \Z}$ that is built in the recent article \cite{HN16} by Hassanzadeh and the second named author. We study the properties of these complexes, of particular complexes $_k^{ \omega}\ZC_\bullet^+,$ where $\omega$ is the canonical module of $R.$ The main results of this section are Propositions~\ref{Proposition3.1},~\ref{Proposition3.3} and~\ref{Proposition3.5}.

\smallskip
In Section~\ref{Section4}, we prove one of the main results of this paper: the Cohen-Macaulayness and the description of the canonical module of residual intersections. Recall that in \cite{Ha12}, Hassanzadeh proved that, under the sliding depth condition, $H_0( _0\ZC_\bullet^+)=R/K$ is Cohen-Macaulay of codimension $s,$ with $K\subset J,\ \sqrt{K}=\sqrt{J},$ and further $K=J$ whenever the residual is arithmetic.  First, we consider the height two case and show that under the $\SD_1$ condition, there exist an epimorphism $\xymatrix{\varphi: H_0(_{s-1}^{\quad \omega}\ZC_\bullet^+)\ar@{->>}[r] &\omega_{R/K}}$ which is an isomorphism if $I$ satisfies $\SD_2$ (Proposition~\ref{Proposition4.4}). By exploring these complexes, we show that, under the $\SD_1$ condition, $K=J;$ and therefore, under the $\SD_2$ condition, the canonical module of $R/J$ is $H_0(_{s-1}^{\quad \omega}\ZC_\bullet^+).$ In a second step, we reduce the general case to the height two case. Our main results in this section  are Theorems~\ref{Theorem4.5} and~\ref{Theorem4.8}.

\smallskip
In Section~\ref{Section5},  we study the stability of Hilbert functions and Castelnuovo-Mumford regularity of residual intersections. Using the acyclicity of $_0\ZC_\bullet^+$, Proposition~\ref{Proposition5.1} says that the Hilbert function of $R/J$ only depends on the degrees of the generators of $\ia$ and the Koszul homologies of $I$.  The graded structure of the canonical module of $R/J$ in Proposition~\ref{Proposition5.3} is the key to derive the  Castelnuovo-Mumford regularity  of residual intersection in Corollary~\ref{Corollary5.4}.

\smallskip
Finally, in Section~\ref{Section6}, we consider the case where $I$ is strongly Cohen-Macaulay. The main results of this section are Theorems~\ref{Theorem6.2} and~\ref{Theorem6.7}. In particular,   for $1\leq k\leq s-g,$ 
\begin{displaymath}
\omega_{\Sym_R^k(I/\ia)}\simeq\Sym_R^{s-g+1-k}(I/\ia)\otimes_R \omega,
\end{displaymath}
whenever $\Tor_1^R(R/I,\omega)=0.$ Consequently, we obtain some tight relations between the Hilbert series of the symmetric powers of $I/\ia$  in Corollary~\ref{Corollary6.5}.  We also give the closed formulas for the types and for the Bass number of some symmetric powers of $I/\ia$ in Corollaries~\ref{Corollary6.6} and~\ref{Corollary6.7}, respectively.


\section{Koszul cycles and approximation complexes}\label{Section2}
In this section we collect the notations and general facts about Koszul complexes and approximation complexes. The reader can consult for instance  \cite[Chapter~1]{BH98} and \cite{SV81, HSV82, HSV1983,HSV83}. We give some results on the duality for Koszul cycles and describe the $0$-th homology modules of approximation complexes with coefficients in a module. 

\smallskip
Assume that $R$ is a Noetherian ring, $I=(f_1,\ldots,f_r)$ is an ideal of $R.$
Let $M$ be a finitely generated $R$-module. The symmetric algebra of $M$ is denoted by $\Sym_R(M)$ and the $k$-th symmetric power of $M$ is denoted by $\Sym_R^k(M).$  We consider  $S=R[T_1,\ldots,T_r]$ as a standard graded algebra over $S_0=R.$ For a graded $S$-module $N,$ the $k$-th graded component of $N$ is denoted by $N_{[k]}.$ We make $\Sym_R(I)$ an $S$-algebra via the graded ring homomorphism $S\longrightarrow \Sym_R(I)$ sending $T_i$ to $f_i$ as an element of $\Sym_R(I)_{[1]}=I,$ and write $\Sym_R(I)=S/\mathfrak{L}.$ 

\smallskip
For a sequence of elements $\xx$ in $R,$ we denote the Koszul complex by $K_\bullet(\xx;M),$ its cycles by $Z_i(\xx;M),$ its boundaries by $B_i(\xx;M)$ and its homologies by $H_i(\xx;M).$ If $M=R,$ then we denote, for simplicity, $K_i, Z_i, B_i, H_i.$ To set more notation, when we draw the picture of a double complex obtained from a tensor product of two complexes (in the sense of \cite[2.7.1]{Wei94}) which at least one of them is finite, say $A\otimes B$ where $B$ is finite, we always put $A$ in the vertical and $B$ in the horizontal one. We also label the module which is in the up-right corner by $(0,0)$ and consider the labels for the rest, as the points in the third-quadrant.

\begin{lem} \label{Lemma2.1}
Let $R$ be a ring  and let $I=(f_1,\ldots,f_r)$ be an ideal of $R.$ If $I=R,$ then $Z_i\simeq \bigwedge\limits^i R^{r-1}.$
\end{lem}
\begin{proof}
Since $I=R,\ H_i=0,$ for all $i$ by \cite[Proposition~1.6.5(c)]{BH98}. The result follows from the fact that the Koszul complex is split exact in this case.
\end{proof}

Let us recall the conditions $\Ser_k$ of Serre. Let $R$ be a Noetherian ring, and $k$ a non-integer. A finitely generated $R$-module $M$ satisfies \emph{Serre's condition} $\Ser_k$ if 
$$\depth(M_\point) \geq \min\{ k,\dim M_\point\}$$
for every prime ideal $\point$ of $R.$

\smallskip
Let $(R,\m)$ be a Noetherian local ring. The local cohomology modules of an $R$-module $M$ are denoted by $$H_\m^i(M): \, H_\m^i(M)=\lim_{\rightarrow}\Ext_R^i(R/m^n,M).$$
The local cohomoly functors $H_\m^i$ are the right-derived functors of $H_\m^0.$ The local cohomology can also be computed with the {\v C}ech complex $C_\m^\bullet$ constructed on a parameter system of $R:\, H_\m^i(M)=H^i(M\otimes_R C_\m^\bullet).$

\smallskip
Duality results for Koszul homology modules over Gorenstein rings have been obtained by several authors, for instance in \cite{Herzog74, Chardin04, MRS14}. For Koszul cycles,  the following holds.
\begin{pro}\label{Proposition2.2}
Let $(R,\m)$ be a Noetherian local ring  and let $I=(f_1,\ldots,f_r)$ be an ideal of $R.$ Suppose that $R$ satisfies  $\Ser_2$ and $\hei(I)\geq 2.$ Then, for all $0\leq i\leq r-1,$
\begin{align*}
Z_i\simeq \Hom_R(Z_{r-1-i},R). 
\end{align*}
\end{pro}
\begin{proof}
The inclusions $Z_i\hookrightarrow K_i=\bigwedge\limits^iR^r$ and $Z_{r-1-i}\hookrightarrow K_{r-1-i}=\bigwedge\limits^{r-1-i}R^r$ induce a map
\begin{align*}
\xymatrix{\varphi_i:Z_i\times Z_{r-1-i}\ar[r] & K_i\times K_{r-1-i}\ar[r]& K_{r-1},}
\end{align*}
where the last map is the multiplication of the Koszul complex, which is a differential graded algebra, and $\Im(\varphi_i)\subset Z_{r-1}\simeq K_r\simeq R.$ It follows that $\varphi_i$ induces a map
\begin{align*}
\xymatrix{\psi_i:Z_i\ar[r] & \Hom_R(Z_{r-1-i},R).}
\end{align*}

\smallskip
We induct on the height to show that for every $\point\in \Spec(R),\ (\psi_i)_\point$ is an isomorphism.  If $\hei(\point)<2,$ then $I_\point=R_\point$, by Lemma~\ref{Lemma2.1} 
\begin{displaymath}
(Z_i)_\point  \simeq \bigwedge^i R_\point^{r-1}\quad \text{and}\quad  (Z_{r-1-i})_\point\simeq \bigwedge^{r-1-i} R_\point^{r-1}
\end{displaymath}
and \cite[Proposition~1.6.10(b)]{BH98} shows that $(\psi_i)_\point$ is an isomorphism.

\smallskip
Suppose that $\hei(\point)\geq 2$ and $(\psi_i)_\q$ is an isomorphism for all prime contained properly in $\point.$ Replacing  $R$ by $R_\point$ and $\m$ by $\point R_\point,$ we can suppose that $\psi_i$ is an isomorphism on the punctured spectrum : the kernel and the corkernel of $\psi_i$ are annihilated by a power of $\m.$ It follows that $H_\m^j(\Ker(\psi_i))=H_\m^j(\coker(\psi_i))=0$ for $j>~0.$ Since $R$ satisfies $\Ser_2$,  $\depth(Z_i)\geq \min\{2,\depth(R)\}=2.$ The exact sequence
\begin{align*}
0\longrightarrow \Ker(\psi_i) \longrightarrow Z_i \longrightarrow \Im(\psi_i)\longrightarrow 0
\end{align*}
implies that $\Ker(\psi_i)=H_\m ^0(\Ker(\psi_i))=0$. Observing that $$\depth(\Hom_R(Z_{r-1-i}, R))\geq \min\{2,\depth(R) \}= 2,$$
the exact sequence 
\begin{align*}
\xymatrix{0 \ar[r]& Z_i\ar[r] & \Hom_r(Z_{r-1-i},R) \ar[r]& \coker(\psi_i) \ar[r]& 0}
\end{align*}
 implies that  $\coker(\psi_i)=H_\m^0(\coker(\psi_i))=0$.
\end{proof}

\smallskip
To fix the terminology we will use, we recall some notations and definitions. Let $(R,\m)$ be a Noetherian local ring. The injective envelope of the residue field $R/\m$ is denoted by $E(R/\m)$ (or by $E$ when the ring is clearly identified by the context). The Matlis dual of an $R$-module $M$ is the module $M^\vee=\Hom_R(M,E(R/\m)).$  The Matlis duality functor is exact, sends Noetherian modules to Artinian modules and Artinian modules to Noetherian modules, and preserves annihilators. 
 
 \smallskip
 When the module $M$ is finitely generated, we have $M^{\vee\vee} \simeq\widehat{M},$ the $\m$-adic completion of $M,$ while $X\simeq X^{\vee\vee}$ when the module $X$ is of finite length.
 
\smallskip
When $R$ is the homomorphic image of a Gorenstein local ring $A,$ the canonical module of a finitely generated $R$-module $M,$ denoted by $\omega_M,$ is defined by
 $$\omega_M:=\Ext_A^{m-n}(M,A)$$
 where  $m=\dim(A)$ and  $n=\dim(M)=\dim(R/\ann_R(M)).$ This module does not depend on $A.$ By the local duality theorem
 $$H_\m^n(M)\simeq \omega_M^\vee.$$
 
\smallskip
We are particularly interested in the case that $R$ admits the canonical module, hence in the sequel we asume that $R$ is the quotient of a Gorenstein ring and write $\omega$ for the canonical module of $R.$ Whenever $R$ is Cohen-Macaulay, $\omega$ is a canonical module of $R$ in the sense of \cite[Definition~3.3.1]{BH98}.

\smallskip
If $R$ is a Gorenstein local ring, $\omega\simeq R$, therefore, by Proposition~\ref{Proposition2.2}, 
$$\omega_{Z_p}\simeq Z_{r-1-p}$$
for all $0\leq p\leq r-1.$ To generalize this result, we will use a result of Herzog and Kunz,
\begin{lem} \cite[Lemma~5.8]{HK71} \label{Lemma2.3}
Let $(R,\m)$  be a Noetherian local ring and let $M,N$ be two finitely generated $R$-modules. If $\widehat{M} \simeq \widehat{N},$ then  $M\simeq N.$
\end{lem}
We will denote by $Z_i^\omega:=Z_i(\ff;\omega)$ the module of $i$-th Koszul cycle, with $\ff=f_1,\ldots,f_r.$ 
\begin{pro}\label{Proposition2.4}
Let $(R,\m)$ be a Noetherian local ring of dimension $d$ which is an epimorphic image of a Gorenstein ring. Suppose that $I=(f_1,\ldots,f_r)$ is an ideal of $R,$ with $\hei(I)\geq 2.$ Then, for all $0\leq p\leq r-1,$
\begin{displaymath}
\omega_{Z_p}\simeq Z_{r-1-p}^{\omega_R}.
\end{displaymath}
Moreover, if $R$ satisfies $\Ser_2,$ then
\begin{displaymath}
\omega_{Z^{\omega_R}_p}\simeq Z_{r-1-p}.
\end{displaymath}
\end{pro}
\begin{proof} 
For simplicity, set  $\omega :=\omega_R$. First we  consider the truncated complexes
\begin{align*}
\textbf{K}_\bullet^{>p}: 0\longrightarrow K_r\longrightarrow \cdots \longrightarrow K_{p+1} \longrightarrow Z_p\longrightarrow 0.
\end{align*}
The double complex $C_\m^\bullet (\textbf{K}_\bullet^{>p})$ gives rise to two spectral sequences. The second terms of the horizonal spectral are
\begin{align*}
^2\textbf{E}_{\hor}^{-i,-j}= H_\m^j(H_{i+p})
\end{align*}
and the first terms of the vertical spectral are
\begin{small}
\begin{align*}
\xymatrix@R=10pt{0\ar[r]&H_\m^0(K_r)\ar[r]&\cdots \ar[r] &H_\m^0(K_{p+1})\ar[r]&H_\m^0(Z_p)\ar[r]&0\\   &\vdots   &\cdots & \vdots &\vdots &\\
0\ar[r]&H_\m^{d-1}(K_r)\ar[r]&\cdots \ar[r] &H_\m^{d-1}(K_{p+1})\ar[r]&H_\m^{d-1}(Z_p)\ar[r]&0\\
0\ar[r]&H_\m^d(K_r)\ar[r]&\cdots \ar[r] &H_\m^d(K_{p+1})\ar[r]&H_\m^d(Z_p)\ar[r]&0.}
\end{align*}
\end{small}
	
\smallskip
Since $I$ annihilates $H_i,\ \dim(H_i)=\dim(R/I)\leq \dim(R)-\hei(I)\leq d-2$ if $H_i\neq 0.$ Therefore, $^2\textbf{E}_{\hor}^{-i,-j}= H_\m^j(H_{i+p})=0,$ for all $j>d-2.$ The comparison of two spectral sequences gives a short exact sequence
\begin{align}\label{Sequence3.1}
\xymatrix{H_\m^d(K_{p+2}) \ar[r] & H_\m^d(K_{p+1})\ar[r]&H_\m^d(Z_p)\ar[r]& 0.}
\end{align}
	
\smallskip
By local duality
\begin{displaymath}
H_\m^d(K_i)\simeq \Hom_R(K_i,\omega)^\vee\simeq (\Hom_R(K_i,R)\otimes_R\omega)^\vee\simeq (K_{r-i}\otimes_R\omega)^\vee=K_{r-i}(\ff;\omega)^\vee.
\end{displaymath}
	
\smallskip
Thus the exact sequence (\ref{Sequence3.1}) provides an exact sequence
\begin{align*}
\xymatrix{K_{r-p-2}(\ff;\omega)^\vee \ar[r] & K_{r-p-1}(\ff;\omega)^\vee\ar[r]&H_\m^d(Z_p)\ar[r]& 0}
\end{align*}
that gives $H_\m^d(Z_p)\simeq {Z^\omega_{r-1-p}}^\vee.$  Then the first isomorphism follows from this isomorphism, the local duality, and Lemma~\ref{Lemma2.3}.
	
\smallskip
The second assertion is proved similarly, by considering the truncated complexes
\begin{align*}
{\textbf{K}^\omega}_\bullet^{>p}: 0\longrightarrow K_r(\ff;\omega)\longrightarrow \cdots \longrightarrow K_{p+1}(\ff;\omega) \longrightarrow Z^\omega_p\longrightarrow 0 
\end{align*}
and  the double complex $C_\m^\bullet ({\textbf{K}^\omega}_\bullet^{>p}).$ 
	
\smallskip
Since $I$ annihilates $H_i(\ff;\omega),\ \dim(H_i(\ff;\omega))\leq \dim(R)-\hei(I)\leq d-2,$ for all $0\leq i\leq r-2.$ Thus $H_\m^j(H_i(\ff;\omega))=0,$ for all $j>d-2$ and $0\leq i\leq r-2.$ By comparing two spectral sequences, we also obtain a short exact sequence
\begin{align} \label{Sequence3.2}
\xymatrix{H_\m^d(K_{p+2}(\ff;\omega))\ar[r] & H_\m^d(K_{p+1}(\ff;\omega))\ar[r]&H_\m^d(Z^\omega_p)\ar[r]& 0.}
\end{align}
	
\smallskip
By local duality
\begin{align*}
H_\m^d(K_i(\ff;\omega))&\simeq H_\m^d(K_i\otimes_R \omega)\simeq \Hom_R(K_i\otimes_R \omega,\omega)^\vee\\
&\simeq \Hom_R(K_i,\Hom_R (\omega,\omega))^\vee\simeq \Hom_R(K_i,R)^\vee\simeq  K_{r-i}^\vee
\end{align*}
as $\Hom_R(\omega,\omega)\simeq R$ since $R$ satisfies $\Ser_2.$

\smallskip
The exact sequence (\ref{Sequence3.2}) provides an exact sequence
\begin{align*}
\xymatrix{K_{r-p-2}^\vee\ar[r] & K_{r-p-1}^\vee\ar[r]&H_\m^d(Z^\omega_p)\ar[r]& 0}
\end{align*}
which shows that $H_\m^d(Z^\omega_p)\simeq Z_{r-1-p}^\vee.$
\end{proof}

Now we describe the $0$-th homology module of  \textit{approximation complexes.} These complexes was introduced in \cite{SV81} and systematically developed in \cite{HSV82} and \cite{HSV1983}. Recall that the approximation complex $\ZC_\bullet(\ff;M)$ is
\begin{align*}
\xymatrix{0\ar[r]& Z_{r}^M\otimes_R S(-r)\ar[r]&\cdots \ar[r] & Z_{1}^M\otimes_R S(-1)\ar[r]^{\partial^\TT_M}& Z_0^M\otimes_R S\ar[r]&0}
\end{align*}
that can be written
\begin{align*}
\xymatrix{0\ar[r]& Z_{r}^M[\TT](-r)\ar[r]&\cdots \ar[r] & Z_{1}^M[\TT](-1)\ar[r]^{\partial^\TT_M}& M[\TT]\ar[r]&0}
\end{align*}
where $ \TT=T_1,\ldots,T_r$ and $Z_i^M=Z_i(\ff;M)$ is the $i$-th Koszul cycle of $K_\bullet(\ff;M).$ By the definition,
\begin{align} \label{Formula2}
H_0(\ZC_\bullet(\ff;M)) \simeq M[T_1,\ldots,T_r]/\mathfrak{L}_M,
\end{align}
where $\mathfrak{L}_M$ is the submodule of $M[T_1,\ldots,T_r]$ generated by the linear forms $c_1T_1+\cdots+c_rT_r$ with $(c_1,\ldots,c_r)\in Z_1^M.$  

\smallskip
Let $\F_\bullet$ be a free resolution of $R/I$ of the form
\begin{align*}
\xymatrix{\cdots \ar[r]& F_1\ar[r]^\delta&R^r \ar[r] & R\ar[r]&0},
\end{align*}
where $F_1$ is the free $R$-module indexed by a generating set of $Z_1.$ By the definition,
$$\Tor_1^R(R/I,M)=Z_1^M/\Im(\delta\otimes 1_M)\hookrightarrow M^r/\Im(\delta\otimes 1_M),$$
where $1_M$ denote the identity morphism on $M.$ Note that $\delta$ is induced by the inclusion $\xymatrix{Z_1\ar@{^{(}->}[r] &R^r.}$ Therefore, $\Im(\delta\otimes 1_M)=(\delta \otimes 1_M)(Z_1\otimes_R M)$ and we obtain an exact sequence
\begin{align}\label{Sequence3.3}
\xymatrix{ Z_1\otimes_R M\ar[rr]^{\delta\otimes 1_M}&& Z_1^M\ar[r] & \Tor_1^R(R/I,M)\ar[r]& 0.}
\end{align}

\smallskip
Let $\mathfrak{L}$ be the submodule of $S=R[T_1,\ldots,T_r]$ generated by the linear forms $c_1T_1+\cdots+c_rT_r$ with $(c_1,\ldots,c_r)\in Z_1.$ Then the exact sequence
\begin{displaymath}
	\xymatrix{0\ar[r] & \mathfrak{L} \ar[r]^\theta & S\ar[r]  & \Sym_R(I)\ar[r] & 0}
\end{displaymath}
provides an exact sequence
\begin{displaymath}
	\xymatrix{\cdots \ar[r] & \mathfrak{L}\otimes_R M \ar[r]^{\theta\otimes 1_M\quad} & M[T_1,\ldots,T_r]\ar[r]  & \Sym_R(I)\otimes_R M\ar[r] & 0.}
\end{displaymath}
The image of $\theta\otimes 1_M$ is denoted by $\mathfrak{L}M.$ It follows that 
\begin{align}\label{Formula3}
	\Sym_R(I)\otimes_R M\simeq M[T_1,\ldots,T_r]/\mathfrak{L}M.
\end{align}
Notice that $\mathfrak{L}M$ is the submodule of $M[T_1,\ldots,T_r]$ generated by the linear forms $c_1T_1+\cdots+c_rT_r$ with $(c_1,\ldots,c_r)\in \Im(\delta\otimes 1_M)\subset Z_1^M$; thus $\mathfrak{L}M\subset \mathfrak{L}_M.$ 

\smallskip
Let $\mathfrak{L'}$ be the submodule of $M[T_1,\ldots,T_r]/\mathfrak{L}M$ generated by the linear forms $c_1T_1+\cdots+c_rT_r+\mathfrak{L}M$ with $(c_1,\ldots,c_r)+\Im(\delta\otimes 1_M)\in \Tor_1^R(R/I,M).$ Then $$\mathfrak{L'}\simeq \mathfrak{L}_M/\mathfrak{L}M.$$

It follows that
\begin{align*}
M[T_1,\ldots,T_r]/(\mathfrak{L}M+\mathfrak{L'})&\simeq (M[T_1,\ldots,T_r]/\mathfrak{L}M)/ \mathfrak{L'} \\
&\simeq  (M[T_1,\ldots,T_r]/\mathfrak{L}M)/(\mathfrak{L}_M/\mathfrak{L}M) \\
&\simeq M[T_1,\ldots,T_r]/\mathfrak{L}_M.
\end{align*}
Thus we have already proved the following.
\begin{pro}\label{Proposition2.5}
Let $R$ be a Noetherian ring and let $I=(f_1,\ldots,f_r)$ be an ideal of $R. $ Assume that $M$ is a finitely generated $R$-module. Then
$$
H_0(\ZC_\bullet(\ff; M)) \simeq M[T_1,\ldots,T_r]/(\mathfrak{L}M+\mathfrak{L'}),
$$
where $\mathfrak{L}\subset S$ is the defining ideal of $\Sym_R(I)$ and $\mathfrak{L'}$ is spanned by generators of $\Tor_1^R(R/I,M).$
\end{pro}

\begin{pro}  \label{Proposition2.6}
Let $R$ be a Noetherian ring and let $I=(f_1,\ldots,f_r)$ be an ideal of $R. $ Assume that $M$ is a finitely generated $R$-module. Then there exists a narural epimorphism
\begin{align*}
\xymatrix{\varphi: \Sym_R(I)\otimes_R M \ar@{->>}[rr] &&H_0(\ZC_\bullet(\ff;M)), }
\end{align*}
that equals $H_0(\ZC_\bullet(\ff;R))\simeq \Sym_R(I)$ when $M=R.$
Furthermore, $\varphi$ is an isomorphism if and only if $\Tor_1^R(R/I,M)=0.$
\end{pro}
\begin{proof}
As $\mathfrak{L}M\subset \mathfrak{L}_M,$ we can define an epimorphism 
\begin{align*}
\xymatrix{\varphi: \Sym_R(I)\otimes_R M  \ar@{->>}[rr] &&  H_0(\ZC_\bullet(\ff;M)),}
\end{align*}
by \eqref{Formula2} and \eqref{Formula3}. Moreover, the kernel of $\varphi$ is isomorphic to $\mathfrak{L}_M/\mathfrak{L}M.$  Thus $\Tor_1^R(R/I,M)=0$ if and only if  $\varphi$ is an isomorphism. 
\end{proof}


\section{Residual approximation complexes} \label{Section3}
Assume that $R$ is a Noetherian ring of dimension $d,\ I=(\ff)=(f_1,\ldots,f_r)$ is an ideal of height $g.$ Let $\ia = (a_1,\ldots, a_s)$ be an ideal contained in $I$ with  $s\geq g.$ Set $J=\ia\colon_RI,\ S=R[T_1,\ldots,T_r]$ and  $\ig:= (T_1,\ldots,T_r).$ We write $a_i=\sum_{j=1}^{r}c_{ji}f_j,$ and $\gamma_i=\sum_{j=1}^{r}c_{ji}T_j.$ Notice that the $\gamma_i$'s depend on how one expresses the $a_i$'s as a linear combination of the $f_i$'s.  Set $\gamma=\gamma_1,\ldots,\gamma_s.$ Finally, for a graded module $N,$ we define $\ennd(N):=\sup \{\mu \mid N_\mu\neq 0\}$ and $\indeg(N):=\inf \{\mu \mid N_\mu\neq 0\}.$

\smallskip
Let $M$ be a finitely generated $R$-module.  We denote by  $\ZC_\bullet(\ff;M)$  the approximation complex associated to $\ff$ with coefficients in $M$ and by $K_\bullet(\gamma;S)$ the Koszul complex associated to $\gamma$ with coefficients in $S.$  Let $\D_\bullet^M=\Tot(\ZC_\bullet(\ff;M)\otimes_S K_\bullet(\gamma;S)).$ Then,
\begin{equation*}
\D_i^M=\bigoplus_{j=i-s}^{i}(Z_j^M\otimes_R S)^{\binom{s}{i-j}}(-i),
\end{equation*}
with  $Z_j^M=0$ for $j<0$ or $j>r$, and for $j=r$ unless $\depth_M(I)=0$. 

\smallskip
In what follows, we assume that $\depth_M(I)>0,$ (hence $Z_r^M =0$), in order that the complexes we construct have length $s.$

\smallskip
We recall that the $k$-th graded component of a graded $S$-module $N$ is denoted by $N_{[k]}.$ We have $(\D_i^M)_{[k]}=0$ for all $k<i.$ Consequently, the complex $(\D_\bullet^M)_{[k]}$ is
\begin{align*}
\xymatrix{0\ar[r]& (\D_k^M)_{[k]}\ar[r]&(\D_{k-1}^M)_{[k]}  \ar[r] & \cdots\ar[r]& (\D_0^M)_{[k]}\ar[r]&0}.
\end{align*}

\smallskip
The \v Cech complex of $S$
with respect to the ideal $\ig=(T_1,\ldots, T_r)$ is denoted by $C_\ig^\bullet=C_\ig^\bullet(S).$

\smallskip

We now consider the double complex $C_\ig^\bullet \otimes_S \D_\bullet^M$ that gives rise to two spectral sequences. The second terms of the horizonal spectral are
$$^2\textbf{E}_{\hor}^{-i,-j}=H_\ig^j(H_i(\D_\bullet^M))$$
and the first terms of the vertical spectral are
\begin{small}
\begin{align*}
^1\textbf{E}_{\ver}^{-\bullet ,-j}=\begin{cases}
\xymatrix{0\ar[r]& H_\ig^r(\D_{r+s-1}^M)\ar[r]&\cdots \ar[r] & H_\ig^r(\D_1^M)\ar[r]&H_\ig^r(\D_0^M)\ar[r]& 0}\; \text{if} \; j=r\\ \; 0 \quad \text{otherwise}
\end{cases}
\end{align*}
\end{small}
and
\begin{equation*}
H_\ig^r(\D_i^M)\simeq \bigoplus_{j=i-s}^{i}(Z_j^M\otimes_R H_\ig^r(S))^{\binom{s}{i-j}}(-i)
\end{equation*}
by \cite[Lemma~2.1]{Ha12}. Since $\ennd (H_\ig^r(S))=-r,$  it follows that $\ennd(H_\ig^r(\D_i^M))=i-r$ if $\D_i^M \not= 0$, thus $H_\ig^r(\D_i^M)_{[i-r+j]}=0,$ for all $j\geq 1.$ Hence, the $k$-th graded component of $^1\textbf{E}_{\ver}^{-\bullet ,-r}$ is the complex:
\begin{align*}
\xymatrix{0\ar[r]& H_\ig^r(\D_{r+s-1}^M)_{[k]}\ar[r]&\cdots \ar[r] & H_\ig^r(\D_{r+k+1}^M)_{[k]}\ar[r]&H_\ig^r(\D_{r+k}^M)_{[k]}\ar[r]& 0.}
\end{align*}

\smallskip
Comparison of the spectral sequences for the two filtrations leads to the definition of the complex of lenght $s$:
\begin{align*}
_k^M\!\ZC_\bullet^+: \xymatrix{0\ar[r]& _k^M\!\ZC_s^+\ar[r]&\cdots \ar[r] & _k^M\!\ZC_{k+1}^+\ar[r]^{\tau_k}&_k^M\!\ZC_{k}^+\ar[r]&\cdots\ar[r]&_k^M\!\ZC_0^+ \ar[r]&0}
\end{align*}
wherein
\begin{equation*} 
_k^M\!\ZC_i^+=\begin{cases}
(\D_i^M)_{[k]}\qquad\quad \quad i\leq \min\{k,s\}\\ 
H_\ig^r(\D_{r-1+i}^M)_{[k]}\;\quad i>k.
\end{cases}
\end{equation*}
and the morphism $\tau_k$ is defined through the transgression. Notice that $_k^M\!\ZC_\bullet^+$ is a direct generalization  of the complex  $_k\ZC_\bullet^+$ in \cite[Section 2.1]{HN16}. 

\smallskip
Since $H_\ig^r(M\otimes_R S)\simeq M\otimes_R H_\ig^r(S)$ for any $R$-module $M, \; _k^M\!\ZC_\bullet^+$ have, like graded strands of $\D_\bullet^M,$ components that are direct sums of Koszul cycles of $K_\bullet(\ff;M).$

\smallskip
The structure of $_k^M\!\ZC_\bullet^+$ is depending on the generating sets of $I,$ on the expression of the generators of $\ia$ in terms of the generators of $I$ and on $M.$ The complex $_k^R\!\ZC_\bullet^+$ considered by Hassanzadeh and the second named author in \cite{HN16}, will be denoted by $_k\ZC_\bullet^+$ instead of $_k^R\!\ZC_\bullet^+.$ 
\begin{defi}
The complex $_k^M\!\ZC_\bullet^+$ is called the $k$-th \textit{residual approximation complex} of $J=\ia\colon_R I$ with coefficients in $M.$
\end{defi}

We consider the morphism 
\begin{align*}
\xymatrix{ M[T_1,\ldots, T_r]^s(-1)\simeq M\otimes_R S^s(-1)\ar[rr]^{\qquad 1_M\otimes \partial_1^\gamma }&& M\otimes_R S\simeq M[T_1,\ldots, T_r],}
\end{align*}
where $\partial_1^\gamma$ is the first differential of $K_\bullet(\gamma;S),$
and denote by $\gamma M$ the image of $1_M\otimes \partial_1^\gamma.$ It is the submodule of $M[T_1,\ldots,T_r]$ generated by the linear forms $\gamma_1,\ldots,\gamma_s.$ Recall from Section~\ref{Section2} that we set $\mathfrak{L}$ for the defining ideal of $\Sym_R(I)$ in $S$ and $\mathfrak{L'}$ for the module spanned by the linear forms correspond to generators of $\Tor_1^R(R/I,M).$
\begin{pro} \label{Proposition3.1}
 Let $R$ be a Noetherian ring and let $\ia\subset I$ be two ideals of $R.$ Suppose that $M$ is a finitely generated $R$-module. Then
 $$H_0(\D_\bullet^M)\simeq M[T_1,\ldots,T_r]/(\mathfrak{L}M+\mathfrak{L'} + \gamma M)$$
 and for all $k\geq 1,$ 
$$H_0(_k^M\!\ZC_\bullet^+) \simeq M[T_1,\ldots,T_r]_{[k]}/(\mathfrak{L}M+\mathfrak{L'} + \gamma M)_{[k]}.$$
\end{pro}
\begin{proof}
The first isomorphism follows from the definition of $\D_\bullet^M$ and Proposition~\ref{Proposition2.5}. The last isomorphism is a consequence of the fact that, for all $k\geq~1,\, H_0(_k^M\!\ZC_\bullet^+)=H_0(\D_\bullet^M)_{[k]}$ is the $k$-th graded component of $H_0(\D_\bullet^M).$
\end{proof}

\begin{pro}\label{Proposition3.3}
Let $R$ be a Noetherian ring and let $\ia\subset I$ be two ideals of $R. $ Assume that $M$ is a finitely generated $R$-module. Then, for all $k\geq 1,$ there exists a natural epimorphism
\begin{align*}
\xymatrix{\psi: \Sym_R^k(I/\ia)\otimes_R M \ar@{->>}[rr] &&H_0(_k^M\!\ZC_\bullet^+). }
\end{align*}
Furthermore, $\psi$ is an isomorphism if $\Tor_1^R(R/I,M)=0.$
\end{pro}
\begin{proof}
As $\Sym_R(I/\ia)\simeq \Sym_R(I)/\ia\Sym_R(I)\simeq S/(\mathfrak{L}+(\gamma)),$  we have an exact sequence
\begin{displaymath}
\xymatrix{ \mathfrak{L} \oplus (\gamma)\ar[r]^{\quad\alpha }& S\ar[r]  & \Sym_R(I/\ia)\ar[r] & 0}
\end{displaymath}
which provides  a commutative diagram, with exact rows
\begin{displaymath}
\xymatrix{ (\mathfrak{L}\oplus(\gamma))\otimes_R M \ar[rr]^{\quad\qquad\alpha\otimes 1_M\quad } \ar[d]_{\simeq} && M[\TT]\ar[r] \ar[d]^{=}  & \Sym_R(I/\ia)\otimes_R M\ar[r]\ar[d]^{=} & 0\\
\mathfrak{L}\otimes_R M\oplus(\gamma)\otimes_R M \ar[rr]^{\quad\qquad\theta\otimes 1_M\oplus \beta\otimes 1_M}&& M[\TT]\ar[r]  & \Sym_R(I/\ia)\otimes_R M\ar[r] & 0}
\end{displaymath}
where $\beta$ is the inclusion $(\gamma)\hookrightarrow S$ and hence $\Im(\beta\otimes 1_M) =\gamma M.$ It follows that
$$\Sym_R(I/\ia)\otimes_R M\simeq M[\TT]/\Im(\alpha\otimes 1_M)\simeq M[T_1,\ldots,T_r]/(\mathfrak{L}M + \gamma M).$$

The natural onto map 
\begin{align*}
\xymatrix{  M[T_1,\ldots,T_r]/(\mathfrak{L}M +\gamma M) \ar@{->>}[rr] && M[T_1,\ldots,T_r]/(\mathfrak{L}M+\mathfrak{L'} + \gamma M)  }
\end{align*}
provides an epimorphism, for all $k\geq 1,$
\begin{align*}
\xymatrix{ \psi:\Sym_R^k(I/\ia)\otimes_R M \ar@{->>}[rr] &&H_0(_k^M\!\ZC_\bullet^+)  }
\end{align*}
by Proposition~\ref{Proposition3.1}. Moreover, $\Tor_1^R(R/I,M)=0$ is equivalent to $\mathfrak{L}_M=\mathfrak{L}M.$ Thus  $\psi$ is an isomorphism if $\Tor_1^R(R/I,M)=0.$
\end{proof}

\begin{lem}\label{Lemma3.4}
Let $M$ be a module over a ring $R.$ Suppose that $N$ is a quotient of $M[T_1,\ldots,T_r],$ with $T_i's$ indeterminates of degree 1, by a graded submodule. Then, for all $k\geq 1,$
\begin{displaymath}
\ann_R(N_k) \subset \ann_R(N_{k+1}).
\end{displaymath}
\end{lem}
\begin{proof}
We consider a graded $S$-homomorphism of degree zero
\begin{align*}
\xymatrix@R=6pt{\vartheta: M[T_1,\ldots,T_r]^r(-1)\ar[r]&M[T_1,\ldots,T_r] \\ 	\qquad (g_1,\ldots, g_r) \ar@{|->}[r]& \sum_{i=1}^{r}g_iT_i}. 
\end{align*}
	
\smallskip
Then $\vartheta$ provides the  epimorphisms of $R$-modules, for all $k\geq 1,$
\begin{align*}
\xymatrix{\widetilde{\vartheta}_k: N_k^r\ar@{->>}[r]& N_{k+1}.}
\end{align*}
	
\smallskip
We will show that $\ann_R(N_k)\subset \ann_R(N_{k+1}),$ for all $k\geq 1.$ Let $ a\in \ann_R(N_k)$ and $u\in N_{k+1}.$  We have to show that $au=0.$ Since $\widetilde{\vartheta}_k$ is surjective, there exist $g_1,\ldots,g_r\in N_k$ such that $u=\widetilde{\vartheta}_k(g_1,\ldots,g_r).$ Therefore,
$$au=a\widetilde{\vartheta}_k(g_1,\ldots,g_r)=\widetilde{\vartheta}_k(ag_1,\ldots,ag_r)=\widetilde{\vartheta}_k(0,\ldots,0)=0.$$
\end{proof}

\begin{pro}\label{Proposition3.5}
Let $R$ be a Noetherian ring and let $\ia\subset I$ be two ideals of $R. $ Assume that $M$ is a finitely generated $R$-module. Then $J=\ia\colon_R I$ annihilates  $H_0(_k^M\!\ZC_\bullet^+),$ for all $k\geq 1.$
\end{pro}
\begin{proof}
Fix $k\geq 1.$ As in the proof of Lemma~\ref{Lemma3.4}, the epimorphism $\psi$ in Proposition~\ref{Proposition3.3} implies that
\begin{align} \label{Equation3.1}
\ann_R(\Sym_R^k(I/\ia)\otimes_R M) \subset \ann_R(H_0( _k^M\!\ZC_\bullet^+) ).
\end{align}
On the other hand, one always has 
\begin{align}\label{Equation3.2}
\ann_R(\Sym_R^k(I/\ia)) \subset \ann_R(\Sym_R^k(I/\ia)\otimes_R M).
\end{align} 
Notice that $\Sym_R(I/\ia)\simeq \Sym_R(I)/(\gamma)\Sym_R(I)\simeq S/(\mathfrak{L}+(\gamma)).$ By Lemma~\ref{Lemma3.4}, 
\begin{align} \label{Equation3.3}
J=\ann_R(I/\ia)\subset \ann_R(\Sym_R^k(I/\ia)).
\end{align} 
By \eqref{Equation3.1}, \eqref{Equation3.2} and \eqref{Equation3.3},  $J\subset \ann_R(H_0(_k^M\!\ZC_\bullet^+)).$ 
\end{proof}

However, the structure of $H_0(_0^M\!\ZC_\bullet^+)$ is difficult to determine.
We recall a definition of Hassanzadeh and the second named author in \cite[Definition~2.1]{HN16}.
\begin{defi}
Let $R$ be a Noetherian ring and let $\ia\subset I$ be two ideals of $R.$ The \textit{disguised $s$-residual intersection of $I$ w.r.t. $\ia$} is the unique ideal $K$ such that $H_0(_0\ZC_\bullet^+)=R/K.$
\end{defi}

\smallskip
To make use of the acyclicity of the $_k\ZC_\bullet^+$ complexes, we recall the definition of classes of ideals that meet these requirements. 
\begin{defi} \label{Definition3.7}
Let $(R,\m)$ be a Noetherian local ring of dimension $d$ and let $I=(f_1,\ldots,f_r)$ be an ideal of height $g.$ Let $k\geq 0$ be an integer. Then
\begin{enumerate}
\item [(i)]  $I$ satisfies \textit{the sliding depth condition,} $\SD_k,$ if 
$$\depth(H_i(\ff;R))\geq \min\{d-g,d-r+i+k\}, \forall i ;$$
also $\SD$ stands for $\SD_0;$
\item [(ii)] $I$ satisfies  \textit{the sliding depth condition on cycles,} $\SDC_k,$  if $$\depth(Z_i(\ff;R))\geq \min\{d-r+i+k, d-g+2,d\} , \forall i\leq r-g ;$$
\item [(iii)] $I$ is \textit{strongly Cohen-Macaulay} if $H_i(\ff;R)$ is Cohen-Macaulay, for all $i.$
\end{enumerate} 
\end{defi}

\smallskip 
 Clearly $I$ is strongly Cohen-Macaulay if and only if $I$ satisfies $\SD_t,$ for all $t\geq r-g.$ Some of the basic properties and relations between such conditions $\SD_k$ and $\SDC_k$ are given in  \cite[Remark~2.4, Proposition~2.5]{Ha12}, \cite[Proposition~2.4]{HN16}, also see \cite{HSV83, HVV85, Va94}. It will be of importance to us that $\SD_k$ implies $\SDC_{k+1}$ whenever $R$ is a Cohen-Macaulay local ring by \cite[Proposition~2.5]{Ha12}.

\begin{rem}\label{Remark3.8}
Notice that adding an indeterminate $x$  to the ring and to ideals $I$ and $\ia.$ One has $(\ia+(x))\colon (I+(x))=(\ia \colon I) +(x)$ in $R[x]$ and in its localization at $\m+(x).$ Hence, for most statements, one may reduce to the case where the height of $I$ is big enough, if needed. 
\end{rem}

In the recent article \cite[Theorem~2.6]{HN16},  Hassanzadeh and the second named author proved the following results. The Cohen-Macaulay hypothesis in this theorem is needed to show that if for an $R$-module $M,\ \depth(M)\geq d-t$ then for any prime $\point,\ \depth(M_\point)\geq \hei(\point)-t,$ see \cite[Section~3.3]{Va94}. 

\begin{theo} \label{Theorem3.9}
Let $(R,\m)$ be a Cohen-Macaulay local ring of dimension $d$ and let $I=(f_1,\ldots,f_r)$ be an ideal of height $g.$ Let $s\geq g$ and fix $0\leq k\leq \min\{s,s-g+2\}.$ Suppose that one of the following hypotheses holds:
\begin{enumerate}
\item[\em (i)] $r+k\leq s$ and $I$ satisfies $\SD,$ or
\item[\em (ii)] $r+k\geq s+1,\ I$ satisfies $\SD$ and $\depth(Z_i)\geq d-s+k,$ for $0\leq i\leq k,$ or
\item[\em (iii)] $I$ is strongly  Cohen-Macaulay.
\end{enumerate}
Then for any $s$-residual intersection $J=(\ia \colon_R I),$  the complex $_k\ZC_\bullet^+$ is acyclic. Furthermore, $\Sym_R^k(I/\ia),$ for $1\leq k\leq s-g+2,$ and the disguised residual intersection $R/K$ are Cohen-Macaulay of codimension $s.$
\end{theo}

Notice that the condition (iii) is stronger than (i) and (ii). In \cite[Theorem~2.11]{Ha12}, Hassanzadeh showed that, under the sliding depth condition $\SD, \, K\subset J$  and $\sqrt{K}=~\sqrt{J}$, and further $K=J,$ whenever the residual is arithmetic.

\section{Cohen-Macaulayness and canonical module of residual intersections}
\label{Section4}

In this section we will prove two important conjectures in the theory of residual intersections: the Cohen-Macaulayness of the residual intersections and the description of their canonical module.

\smallskip
In order to make reduction to lower height case and prove the Cohen-Macaulayness when $s=g,$ we first state the following proposition, which is a trivial generalization of \cite[Lemma~3.5]{HVV85} that only treated the sliding depth condition $\SD.$ The proof goes along the same lines.
\begin{pro}\label{Proposition4.1}
Let $(R,\m)$ be a Cohen-Macaulay local ring, let $I$ be an ideal of height $g$ and $k\geq 0$ be an integer. Let $x_1,\ldots,x_\ell$ be a regular sequence in $I.$ Let $\prime$ denote the canonical epimorphism $R\longrightarrow  R^\prime=R/(x_1,\ldots,x_\ell).$ Then $I$ satisfies $\SD_k$ if and only if $I^\prime$ satisfies $\SD_k$ (in $R^\prime)$. In particular, $I$ is strongly Cohen-Macaulay  if and only if $I^\prime$  is strongly Cohen-Macaulay .
\end{pro}
\begin{pro}\label{Proposition4.2}
Let $(R,\m)$ be a Cohen-Macaulay local ring of dimension $d$ and let $I$ be an ideal of height $g.$ Let $\xx=x_1,\ldots, x_g$ be a regular sequence contained in $I$ and $J=((\xx)\colon_R I).$ Suppose that $R/I$ is Cohen-Macaulay and $I$ satisfies $\SD.$ Then $R/J$ is Cohen-Macaulay of codimension $g.$
\end{pro}
\begin{proof}
The proof goes along the same lines as in \cite{HVV85}. By Proposition~\ref{Proposition4.1}, we may reduce modulo $\xx=x_1,\ldots,x_g$ and consider $R^\prime=R/(\xx).$ Thus we can assume that $\hei(I)=g=0$ and $J=(0\colon_R I).$
	
\smallskip
Suppose that  $I$ is an ideal generated by the sequence $x_1,\ldots,x_r.$ Then $Z_r=(0\colon_R I)=J$ and $K_r\simeq R.$ The exact sequence
\begin{align*}
0\longrightarrow Z_r\longrightarrow K_r \longrightarrow B_{r-1}\longrightarrow 0
\end{align*}
shows that $B_{r-1}\simeq K_r/Z_r\simeq R/J.$
	
\smallskip
Since $I$ satisfies $\SD,\ I$ satisfies $\SDC_1$ by \cite[Proposition~2.5]{Ha12}. It follows that $Z_{r-1}$ is Cohen-Macaulay of dimension $d.$  Moreover, $I$ satisfies $\SD,$ $\depth(H_{r-1})\geq d-1.$ Therefore, the exact sequence
\begin{align*}
0\longrightarrow B_{r-1}\longrightarrow Z_{r-1} \longrightarrow H_{r-1}\longrightarrow 0
\end{align*}
implies that $H_\m^i(B_{r-1})=0,$ for all $i\neq d,$ hence $B_{r-1}$ is Cohen-Macaulay of dimension $d.$
\end{proof}

To study the Cohen-Macaulayness of residual intersections in the general case, we will use the following lemma.
\begin{lem} \label{Lemma4.3}
Let $(R,\m)$ be a Cohen-Macaulay  local ring of dimension $d,$ with canonical module $\omega.$ Suppose that $S=R[T_1,\ldots,T_r]$ is the standard graded polynomial ring over  $R$ and $\ig:=S_{+} .$ Let $\ia\subset I=(f_1,\ldots,f_r)$ be two ideals of $R,$ with $\hei(I)=g.$ If $J=(\ia\colon_R I)$ is an $s$-residual intersection of $I,$  then
\begin{enumerate}
\item[\em (i)] There is a natural graded isomorphism
$$H_\ig^r(S)\simeq  \, ^\ast\!\Homgr_S(S(-r),R).$$
 In particular, for all $\mu\in \Z,$
 \begin{align*}
 H_\ig^r(S)_\mu\simeq S_{-\mu-r}^\ast=\Hom_R(S_{-\mu-r},R).
 \end{align*}
\item[\em (ii)] If $g\geq 2,$ then $\depth(_k\ZC_0^+)=\depth(_k\ZC_s^+)=d,$ for all $0\leq k\leq s-1.$
\item[\em (iii)] If $g=2$ and $I$ satisfies $\SD_\ell,$ then
\begin{align*}
\depth(_0\ZC_i^+)\geq\min\{d, d-s+i+\ell\},
\end{align*}
for all $1\leq i\leq s-1.$ 
\item[\em(iv)] If $g\geq 2,$ then the following diagram, where the vertical isomorphisms are induced by the identifications $H_\m^d(Z_\ast)\simeq {Z^\omega_{r-1-\ast}}^\vee$ in Proposition~\ref{Proposition2.4}, is commutative, for all $0\leq k\leq s-2,$
\begin{displaymath}
\xymatrix{  H_\m^d(_k\ZC_s^+) \ar[r]\ar[d]_{\simeq} & H_\m^d(_k\ZC_{s-1}^+) \ar[d]^{\simeq } \\ 
(_{s-k-1}^{\qquad \omega}\ZC_0^+)^\vee\ar[r] &  (_{s-k-1}^{\qquad \omega}\ZC_1^+)^\vee.}
\end{displaymath}
\end{enumerate}
\end{lem}
\begin{proof} 
(i) is the graded local duality theorem.

\smallskip
(ii)  Since $Z_{r-1}\simeq Z_0=R,\ \depth(_k\ZC_0^+)= \depth(_k\ZC_s^+)=d.$ 

\smallskip
(iii)  By \cite[Proposition~2.5]{Ha12}, $I$ satisfies $\SDC_{\ell+1},$  that is
\begin{displaymath}
\depth(Z_j)\geq \min\{d-r+j+\ell+1,d\},
\end{displaymath}
for all $ 0\leq j\leq r-2.$ 

\smallskip 
For any $1\leq i\leq s-1,$
\begin{equation*}
_0\ZC_i^+=H_\ig^r(\D_{r-1+i})_{[0]}=\bigoplus_{j=r-1+i-s}^{r-1}\big(Z_j\otimes_RH_\ig^r(S)\big)_{[-r+1-i]}^{\binom{s}{r-1+i-j}}.
\end{equation*}
Thus $_0\ZC_i^+$ is a direct sum of copies of  modules  $Z_\delta,\ldots,Z_{r-1},$ where $\delta=\max\{0,r-1+i-s\}.$ Notice that $0\leq \delta\leq r-2.$ It follows that
\begin{align*}
\depth(_0\ZC_i^+)&=\min_{\delta \leq j\leq r-1}\{\depth(Z_j)\} =\min\{\min_{\delta \leq j\leq r-2}\{\depth(Z_j)\},d \}\\
& \geq \min\{d,d-r+\delta+\ell +1\}\geq \min\{d, d-s+i+\ell\}.
\end{align*}

\smallskip
(iv) We have the following commutative diagrams, for all $0\leq k\leq s-2,$
\begin{footnotesize}
\begin{displaymath}
\xymatrix@C=8pt{  H_\m^d(_k\ZC_s^+)=H_\m^d(H_\ig^r(\D_{r+s-1}))_{[k]} \ar[r]\ar[d]_{\simeq} & H_\m^d(H_\ig^r(\D_{r+s-2}))_{[k]}=H_\m^d(_k\ZC_{s-1}^+) \ar[d]^{\simeq } \\ 
H_\m^d(Z_{r-1})\otimes_R H_\ig^r(S)_{[k-r-s+1]} \ar[r] \ar[d]_{\simeq}  & H_\m^d(Z_{r-1})\otimes_R H_\ig^r(S)_{[k-r-s+2]}^s \oplus H_\m^d(Z_{r-2})\otimes_R H_\ig^r(S)_{[k-r-s+2]} \ar[d]^{\simeq} \\
{Z^\omega_0}^\vee \otimes_R S_{[s-k-1]}^\ast\ar[r] \ar[d]_{\simeq}& {Z^\omega_0}^\vee \otimes_R (S_{[s-k-2]}^s)^\ast \oplus {Z^\omega_1}^\vee \otimes_R  S_{[s-k-2]}^\ast\ar[d]^{\simeq}\\
(Z^\omega_0\otimes_R S_{[s-k-1]})^\vee \ar[r] \ar[d]_{\simeq}& (Z^\omega_0\otimes_R S_{[s-k-2]}^s \oplus Z^\omega_1\otimes_R S_{[s-k-2]})^\vee \ar[d]^{\simeq} \\
(_{s-k-1}^{\qquad \omega}\ZC_0^+)^\vee=((\D^\omega_0)_{[s-k-1]})^\vee \ar[r] &  ((\D^\omega_1)_{[s-k-1]})^\vee=(_{s-k-1}^{\qquad \omega}\ZC_1^+)^\vee}
\end{displaymath}
\end{footnotesize}
where the first diagram and the last diagram are commutative by the definitions, the second diagram is commutative by the natural isomorphisms in item (i) and Proposition~\ref{Proposition2.4}, and the third diagram is commutative by the natural isomorphism  
\begin{displaymath}
{Z^\omega_i}^\vee \otimes_R S_{[\ell]}^\ast  \simeq (Z^\omega_i \otimes_R S_{[\ell]})^\vee,
\end{displaymath} for all $i,\ell,$ see \cite[II, \S4, no 4, Proposition 4]{Bou70}.
\end{proof}

\begin{pro} \label{Proposition4.4}
Let $(R,\m)$ be a Cohen-Macaulay local ring of dimension $d,$ with canonical module $\omega,$ and $I=~(f_1,\ldots,f_r)$ be an ideal of height 2. Suppose that $J=(\ia\colon_R I)$ is an $s$-residual intersection of $I$ and $K$ is the disguised $s$-residual intersection of $I$ w.r.t. $\ia.$ If $I$  satisfies $\SD_1,$ then there exists an epimorphism of $R$-modules
\begin{align*}
\xymatrix{\phi: H_0( _{s-1}^{\quad \omega}\ZC_\bullet^+) \ar@{->>}[r]& \omega_{R/K}}
\end{align*}
and  $\phi$ is an isomorphism if $I$ satisfies $\SD_2.$
\end{pro} 

\begin{proof}  Since $I$ satisfies $\SD_1, \ _0\ZC_\bullet^+$ is acyclic and $R/K$ is Cohen-Macaulay of dimension $d-s$ by Theorem~\ref{Theorem3.9}.  By local duality
\begin{equation}\label{Equation4}
H_\m^{d-s}(R/K)\simeq \omega_{R/K}^\vee.
\end{equation}
	
\smallskip
Now the double complex $C_\m^\bullet (_0\ZC_\bullet^+)$ gives rise to two spectral sequences. The second terms of the horizonal spectral are
\begin{align*}
^2\textbf{E}_{\hor}^{-i,-j}=
\begin{cases}
H_\m^{d-s}(R/K)&\text{if} \quad j=d-s\; \text{and}\  i=0\\0 & \text{otherwise}
\end{cases}
\end{align*}
and the first terms of the vertical spectral are
\begin{small}
\begin{align*}
\xymatrix@C=8pt{ 0&0&0&\cdots &0 &0&0\ar@{--}[dddddllllll]\\
0&0&0&\cdots &0 &0&0\\
0& 0 &0&\cdots& 0 & H_\m^{d-s+2}(_0\ZC_1^+)  &0\\
\cdots & \cdots& \cdots& \cdots   & \cdots& \cdots&\cdots \\ 	
0 & 0&H_\m^{d-1}(_0\ZC_{s-2}^+)\ar[r]&\cdots\ar[r] & H_\m^{d-1}(_0\ZC_2^+)\ar[r]& H_\m^{d-1}(_0\ZC_1^+)&0\\
H_\m^d(_0\ZC_s^+)\ar[r]& H_\m^d(_0\ZC_{s-1}^+)\ar[r]&H_\m^d(_0\ZC_{s-2}^+)\ar[r]&\cdots \ar[r] &  H_\m^d(_0\ZC_2^+)\ar[r]& H_\m^d(_0\ZC_1^+)\ar[r]&H_\m^d(R)}
\end{align*}
\end{small}
since $\depth( _0\ZC_i^+)\geq d-s+i+1,$ for all $1\leq i\leq s-1,$ by Lemma~\ref{Lemma4.3}(iii).

\smallskip
By the convergence of the spectral sequences, we obtain
\begin{equation} \label{Equation5}
H_\m^{d-s}(R/K)\simeq \, ^\infty\textbf{E}_{\ver}^{-s,-d}\subset \;  ^2\textbf{E}_{\ver}^{-s,-d}.
\end{equation}

\smallskip
By  Lemma~\ref{Lemma4.3}(iv),  we have the following commutative diagram
\begin{displaymath}
\xymatrix{  H_\m^d(_0\ZC_s^+) \ar[r]\ar[d]_{\simeq} & H_\m^d(_0\ZC_{s-1}^+) \ar[d]^{\simeq } \\ 
	(_{s-1}^{\quad \omega}\ZC_0^+)^\vee\ar[r] &  (_{s-1}^{\quad \omega}\ZC_1^+)^\vee.}
\end{displaymath}
Therefore
\begin{equation} \label{Equation6}
^2\textbf{E}_{\ver}^{-s,-d}\simeq H_0( _{s-1}^{\quad \omega}\ZC_\bullet^+)^\vee.
\end{equation}
By (\ref{Equation4}), (\ref{Equation5}) and  (\ref{Equation6}), we can define a monomorphism of $R$-modules by the compositions
\begin{align*}
\xymatrix{\omega_{R/K}^\vee \ar[r]^{\simeq\qquad} & H_\m^{d-s}(R/K )\ar[r]^{\;\simeq }&  ^\infty\textbf{E}_{\ver}^{-s,-d}\ar@{^(->}[r] &  ^2\textbf{E}_{\ver}^{-s,-d} \ar[r]^{\simeq\quad  } &H_0( _{s-1}^{\quad \omega}\ZC_\bullet^+)^\vee}
\end{align*}
which provides an epimorphism
\begin{align*}
\phi: H_0( _{s-1}^{\quad \omega}\ZC_\bullet^+) \longrightarrow \omega_{R/K}.
\end{align*}

\smallskip
If $I$ satisfies $\SD_2,$ then $\depth( _0\ZC_i^+)\geq \min\{d,d-s+i+2\},$ for all $1\leq i\leq s-1,$ by Lemma~\ref{Lemma4.3}(iii). It follows that 
\begin{equation*} 
H_\m^{d-s}(R/K)\simeq \, ^\infty\textbf{E}_{\ver}^{-s,-d}= \, ^2\textbf{E}_{\ver}^{-s,-d}
\end{equation*}
and thus $\phi$ is an isomorphism.
\end{proof}

Now we state our main result that answers the question of Huneke and Ulrich in \cite[Question~5.7]{HU88} and also answers the conjecture of Hassanzadeh and the second named author in \cite[Conjecture~5.9]{HN16}.  

\begin{theo} \label{Theorem4.5}
Let $(R,\m)$ be a Cohen-Macaulay local ring of dimension $d,$ with canonical module $\omega,$ and $\ia\subset I$ be two ideals of $R,$ with $\hei(I)=g\leq s.$ Suppose that $I$ satisfies $\SD_1$ and $J=(\ia\colon_R I)$ is an $s$-residual intersection of $I.$ Then $R/J$ is Cohen-Macaulay of dimension $d-s.$
\end{theo} 
\begin{proof} 
Let $K$ be the disguised $s$-residual intersection of $I$ w.r.t. $\ia.$ Since $I$ satisfies $\SD_1,$ hence $\SD, \ R/K$ is Cohen-Macaulay of dimension $d-s$ by Theorem~\ref{Theorem3.9} and $K\subset J$ by \cite[Theorem~2.11]{Ha12}. The proof will be completed by showing that $J\subset K.$ 
	
\smallskip
We first consider the case where $g=2.$ By Proposition~\ref{Proposition4.4}, there is the epimorphism 
\begin{displaymath}
\phi: H_0( _{s-1}^{\quad \omega}\ZC_\bullet^+) \longrightarrow \omega_{R/K}.
\end{displaymath}
As $R/K$ is Cohen-Macaulay, $\ann_R(\omega_{R/K})=\ann_R(R/K)=K.$ The epimorphism $\phi$ implies that
\begin{displaymath}
\ann_R(H_0( _{s-1}^{\quad \omega}\ZC_\bullet^+) ) \subset \ann_R(\omega_{R/K}) =K.
\end{displaymath}
By Proposition~\ref{Proposition3.5}, $J\subset \ann_R(H_0( _{s-1}^{\quad \omega}\ZC_\bullet^+) ) \subset K.$

\smallskip
We may always reduce to the case $g\geq 2$ by Remark~\ref{Remark3.8}. If $g>2,$ then we can choose a regular sequence $\aa$ of length $g-2$ inside $\ia$ which is a part of a minimal generating set of $\ia.$ Since $R$ is Cohen-Macaulay, by \cite[Theorem~2.1.3]{BH98}, $R/\aa$ is a Cohen-Macaulay local ring of dimension $d-g+2.$ Moreover,  $J/\aa=\ia/\aa \colon I/\aa$ and $\mu(\ia/\aa)=\mu(\ia)-g+2,$  therefore $J/\aa$ is an $(s-g+2)$-residual intersection of $I/\aa$ which is of height 2. Furthermore, $I/\aa$ satisfies $\SD_1$ by Proposition~\ref{Proposition4.1}.  Hence, it follows from the height two case that  $R/J\simeq (R/\aa)/(J/\aa)$ is Cohen-Macaulay of dimension $d-s.$ 
\end{proof}

It follows from the proof of Proposition~\ref{Proposition3.5} that $J \subset \ann_R(\Sym_R^k(I/\ia)),$ for all $k\geq 1.$  Then a natural question is: under what conditions one has 
$$\ann_R(\Sym_R^k(I/\ia))=J?$$
It is known that $\ann_R(\Sym_R^k(I/\ia))=J,$ for all $k\geq ,1$ whenever $J$ is arithmetic in \cite[Corollary~2.8(iv)]{HN16}. The next result answers this question.
\begin{cor} \label{Corollary4.6}
Let $(R,\m)$ be a Cohen-Macaulay local ring of dimension $d,$ with canonical module $\omega,$ and let $\ia\subset I$ be two ideals of $R,$ with $\hei(I)=g.$ Suppose that $J$ is an $s$-residual intersection of $I$ and let $1\leq k\leq s-g+1.$ 
\begin{enumerate}
\item [\em (i)] If $I$ satisfies $\SD_1,$ then $\Sym_R^k(I/\ia)$ is a faithful $R/J$-module.
\item [\em (ii)] If $I$ satisfies strongly Cohen-Macaulay, then $\Sym_R^k(I/\ia)$ is a maximal Cohen-Macaulay faithful $R/J$-module. 
\end{enumerate}
\end{cor} 
\begin{proof}
(i) The proof will be completed by showing that $\ann_R(\Sym_R^{s-g+1}(I/\ia))\subset J.$  As in the proof of Theorem~\ref{Theorem4.5}, it suffices to prove that $\ann_R(\Sym_R^{s-g+1}(I/\ia))\subset J$  in the case $g=2.$ 
The inclusions $\ann_R(\Sym_R^{s-1}(I/\ia))\subset \ann_R(H_0( _{s-1}^{\quad \omega}\ZC_\bullet^+) )\subset K=~J$ are demonstrated in the proofs of Proposition~\ref{Proposition3.5} and of Theorem~\ref{Theorem4.5}.

\smallskip
(ii) follows immediately from Theorem~\ref{Theorem3.9}, Theorem~\ref{Theorem4.5} and the first item.
\end{proof}

The following example shows that the above corollary does not hold for the $(s-g+2)$-th symmetric power of $I/\ia.$
\begin{exem} \cite[Example~2.10]{HN16}
Let $R=\Q[x,y], I=(x,y)$ and $\ia=(x^2,y^2).$ We set $J=\ia\colon_R I.$ Using  {\tt Macaulay2} \cite{Macaulay2}, we see that $J=(x^2, xy, y^2)$ is a 2-residual intersection (a link in this case) of $I$ and 
\begin{align*}
\Sym_R(I/\ia)\simeq R[T_1,T_2]/\big(xT_1,yT_2,-yT_1+xT_2\big).
\end{align*}
Thus a free resolution of $\Sym_R^2(I/\ia)$ is
\begin{align*}
\xymatrix{ 0 \ar[r] & R^3 \ar[r]^N   & R^6 \ar[r]^M & R^3\ar[r] & \Sym_R^2(I/\ia) \ar[r] &0, }
\end{align*}
where 
\begin{align*}
M=\begin{pmatrix} x & 0 &0 & y &0 &0 \\ 0 & x & 0 & 0 & y &0 \\ 0 & 0 & x & 0 &0 & y
		\end{pmatrix} \quad \text{and} \quad  
N=\begin{pmatrix} -y & 0 &0 \\ 0 & -y &0 \\ 0 &0 & -y\\ x &  0 &0\\ 0 & x & 0 \\ 0 & 0 &  x
\end{pmatrix}.
\end{align*}
It follows that
\begin{align*}
\ann_R(\Sym_R^2(I/\ia))=(x,y)\varsupsetneq J.
\end{align*}
\end{exem}

We now give a description the canonical module of residual intersections. 
\begin{theo} \label{Theorem4.8}
Let $(R,\m)$ be a Cohen-Macaulay local ring of dimension $d,$ with canonical module $\omega,$ and let $\ia\subset I$ be two ideals of $R,$ with $\hei(I)=g.$ Suppose that $I$ satisfies $\SD_2, \ \Tor_1^R(R/I,\omega)=0$ and $J=(\ia\colon_R I)$ is an $s$-residual intersection of $I.$ Then the canonical module of $R/J$ is $\Sym_R^{s-g+1}(I/\ia)\otimes_R \omega.$
\end{theo} 

\begin{proof}
We first consider the case where $g=2.$ By Proposition~\ref{Proposition4.4} and Theorem~\ref{Theorem4.5},
\begin{displaymath}
\omega_{R/J}\simeq H_0( _{s-1}^{\quad \omega}\ZC_\bullet^+)\simeq \Sym_R^{s-1}(I/\ia)\otimes_R \omega.
\end{displaymath}
The last isomorphism by Proposition~\ref{Proposition3.3}.

\smallskip 
We may always reduce to the case $g\geq 2$ by Remark~\ref{Remark3.8}. If $g>2,$ then we can choose a regular sequence $\aa$ of length $g-2$ inside $\ia$ which is a part of a minimal generating set of $\ia$ as in the proof of Theorem~\ref{Theorem4.5}. As $\aa\subset I$ is regular on $\omega,$ 
$$\Tor_1^R(R/I,\omega)\simeq \Tor_1^{R/\aa}(R/I, \omega/\aa\omega)=0.$$

\smallskip
Furthermore, observing that the canonical module of $R/\aa$ is $ \omega/\aa\omega,$ it follows from the height two case that 
\begin{align*}
\omega_{R/J}\simeq \Sym_{R/\aa}^{(s-g+2)-1}\Big(\dfrac{I/\aa}{\ia/\aa}\Big)\otimes_{R/\aa} \omega_{R/\aa}\simeq\Sym_R^{s-g+1}(I/\ia)\otimes_R\omega.
\end{align*}
\end{proof}

Notice that the hypothesis $\Tor_1^R(R/I,\omega)=0$ is always satisfied for ideals of finite projective dimension. In particular,  if $R$ is Gorenstein, then $\omega\simeq R,$ hence  $\Tor_1^R(R/I,\omega)\simeq \Tor_1^R(R/I,R)=0,$ therefore  the canonical module of $R/J$ is $(s-g+1)$-th symmetric power of  $I/\ia.$ As a consequence,  the second conjecture in the introduction is proved under the $\SD_2$ condition.
\begin{rem} \label{Remark4.9}
\item [(i)] Under assumptions of Theorem~\ref{Theorem4.8}, but $I$ only satisfies $\SD_1$  instead of $\SD_2.$ Then there exists an epimorphism of $R$-modules
$$\xymatrix{\Sym_R^{s-g+1}(I/\ia)\otimes_R \omega\ar@{->>}[r]& \omega_{R/J}}.$$
\item[(ii)] In the height two  case, by using Propositon~\ref{Proposition4.4}, we could omit the assumption $\Tor_1^R(R/I,\omega)=0$ in Theorem~\ref{Theorem4.8}. In this case, the canonical module of $R/J$ is the $(s-1)$-th graded component of  
$$\omega[T_1,\ldots,T_r]/ (\mathfrak{L}\omega+\mathfrak{L'}+ \gamma\omega)$$
by Proposition~\ref{Proposition3.1} and Theorem~\ref{Theorem4.5}.
\end{rem}

The following example shows that Theorem~\ref{Theorem4.8} does not hold if $I$ only satisfies $\SD$ condition. 
\begin{exem} \cite[Example~2.9]{EU16}
Let $R=k[[x_1,\ldots,x_5]]$ and let $I$ be the ideal of $2\times 2$ minors of the matrix
\begin{align*}
\begin{pmatrix}x_1 &x_2 &x_3 &x_4 \\ x_2 & x_3 &x_4 & x_5
\end{pmatrix}.
\end{align*}
Then $I$ is of height 3. If we take $\ia$ to be the ideal generated by 4 sufficiently general cubic forms in $I,$ then $J=\ia\colon_R I$ is a 4-residual intersection. Using  {\tt Macaulay2} \cite{Macaulay2}, it is easy to see that $I$ satisfies $\SD.$ Moreover, we see that $I^2/\ia I$ requires  20 generators, whereas $\omega_{R/J}$ requires only 16. Thus there is no surjection $\xymatrix{\omega_{R/J}\ar@{->>}[r]& I^2/\ia I,}$ therefore $\omega_{R/J}$ is not isomorphic to $\Sym_R^2(I/\ia).$

\smallskip
Computation of the initial degree  of $\Sym_R^2(I/\ia)$ and $\omega_{R/J}$ shows that there can be no surjection $\xymatrix{\Sym_R^2(I/\ia) \ar@{->>}[r]&\omega_{R/J}}.$ This shows that  $\SD_1$ condition in Remark~\ref{Remark4.9}(i) is necessary.
\end{exem}

Recall that in a Noetherian local ring $(R,\m),$ the \textit{type} of a finitely generated $R$-module $M$ is the dimension of the $R/\m$-vector space $\Ext_R^{\depth(M)}(R/\m,M)$ and it is denoted by $r_R(M)$ or just $r(M).$ The minimal number of generators of the $R$-module $M$ is the dimension of the $R/\m$-vector space $R/\m\otimes_R M$ and it is denoted by $\mu(M).$ Notice that if $M,N$ are two finitely generated $R$-modules, then 
\begin{align*}
\mu(M\otimes_R N)&=\dim_{R/\m}(M\otimes_R N\otimes_R R/\m)\\
& =\dim_{R/\m}(M\otimes_R R/\m \otimes_{R/\m} N\otimes_R R/\m)\\
&=\mu(M)\mu(N).
\end{align*}

\begin{cor} \label{Corollary4.11}
Under the assumptions of Theorem~\ref{Theorem4.8},
\begin{displaymath}
r(R/J)=\binom{\mu(I/\ia)+s-g}{\mu(I/\ia)-1}r(R).
\end{displaymath}
Thus $R/J$ is Gorenstein if and only if $R$ is Goenstein and $\mu (I/\ia )=1$.
\end{cor}
\begin{proof}
Since the canonical module of $R/J$ is $\Sym_R^{s-g+1}(I/\ia)\otimes_R \omega$ by Theorem~\ref{Theorem4.8}, it follows from \cite[Proposition~3.3.11]{BH98} that
\begin{align*}
r(R/J)=\mu(\omega_{R/J})&=\mu(\Sym_R^{s-g+1}(I/\ia)\otimes_R \omega)\\
&=\mu(\Sym_R^{s-g+1}(I/\ia))\mu(\omega)\\
&=\dim_{R/\m}(\Sym_R^{s-g+1}(I/\ia)\otimes_R R/\m)r(R)\\
&=\dim_{R/\m}(\Sym_{R/\m}^{s-g+1}(I/\ia\otimes_R R/\m))r(R).
\end{align*}
Since $I/\ia\otimes_R R/\m$ is a $R/\m$-vector space of dimension $\mu(I/\ia),$
\begin{align*}
\Sym_{R/\m}(I/\ia\otimes_R R/\m)\simeq (R/\m)[Y_1,\ldots,Y_{\mu(I/\ia)}].
\end{align*}
It follows that 
$$r(R/J)=\binom{\mu(I/a)+s-g}{\mu(I/a)-1}r(R).$$
\end{proof}


\section{Stability of Hilbert functions and Castelnuovo-Mumford regularity of residual intersections} \label{Section5}
One is based on the resolution of residual intersections $_0\ZC_\bullet^+,$ from which we could  provide many informations concerning $R/J,$ like the stability of Hilbert functions and the Castelnuovo-Mumford regularity of residual intersections.

\smallskip
First we study the stability of Hilbert functions of residual intersections. We recall the definitions of  the Hilbert function, Hilbert polynomial and Hilbert series, the reader can consult for instance \cite[Chapter~4]{BH98}. Let $M$ be a graded $R$-module whose graded components $M_n$ have finite length, for all $n.$ The numerical function $H(M,-): \Z\longrightarrow \Z$ with $H(M,n)=\length(M_n),$ for all $n\in \Z,$ is the Hilbert function, and $H_M(t):=\sum_{n\in \Z}H(M,n)t^n$ is the Hilbert series of $M.$\\

If $R$ is assumed to be generated over $R_0$ by elements of degree one, that is, $R=R_0[R_1]$ and $M$ is a finitely generated graded $R$-module of dimension $m\geq 1,$  then there exists a polynomial $P_M(X)\in \Q[X]$ of degree $m-1$ such that $H(M,n)=P_M(n)$ for all $n\gg 0.$ This polynomial is called the Hilbert polynomial of $M.$  We can write
\begin{displaymath}
P_M(X)=\sum_{i=0}^{m-1}(-1)^{m-1-i}e_{m-1-i}\binom{X+i}{i}.
\end{displaymath}
Then the \textit{multiplicity} of $M$ is defined to be
\begin{displaymath}
e(M)=\begin{cases} e_0 & \text{if}\quad m>0\\ \length(M) & \text{if}\quad m=0.
\end{cases}
\end{displaymath}

\smallskip
In  \cite{CEU01}, Eisenbud, Ulrich and the first named author restated an old question of Stanley in \cite{St80} asking for which open sets of ideals $\ia$ the Hilbert function of $R/\ia$ depends only on the degrees of the generators $\ia.$ More precisely, they  consider the following two conditions.
\begin{enumerate}
\item [(A1)] Is the Hilbert function of $R/\ia$ is constant on the open set of ideals $\ia$ generated by $s$ forms of the given degrees such that $\hei(\ia\colon_R I)\geq s;$
\item [(A2)] Is the Hilbert function of $R/(\ia\colon_R I)$ is constant on this set.
\end{enumerate}

\smallskip
It is shown in  \cite[Theorem~2.1]{CEU01} that ideals with some sliding depth conditions  in conjunction with $G_{s-1}$ or $G_s$ satisfy these two conditions. In \cite[Proposition~3.1]{HN16}, Hassanzadeh and the second named author proved that if $(R,\m)$ is a Cohen-Macaulay graded local ring of dimension $d$ over an Artinian local ring $R_0$ and if  $\ia \subset I$ are two homogeneous ideals, $I$ satisfies $\SD,$ and $\depth(R/I)\geq d-s,$ then the above condition (A1) is satisfied for any $s$-residual intersection $J=(\ia\colon_R I).$ It follows directly from \cite[Theorem~2.11]{Ha12} and \cite[Proposition~3.1]{HN16} that if $I$ satisfies $\SD,$ then, for any arithmetic $s$-residual intersection $J=(\ia\colon_R I),$ the above condition (A2) is satisfied. 

\smallskip
The next proposition, we will show that the above condition (A2) is satisfied for any residual intersection under $\SD_1$ condition.
\begin{pro} \label{Proposition5.1}
Let $(R,\m)$ be a graded Cohen-Macaulay local ring over an Artinian local ring $R_0$ and $\ia\subset I$ be two homogeneous ideals, with $\hei(I)=g.$ Suppose that $I$ satisfies $\SD_1$ and $J=(\ia\colon_R I)$ is an $s$-residual intersection of $I.$ Then the Hilbert function of $R/J$ satisfies the above condition (A2).
\end{pro}
\begin{proof} 
By Theorem~\ref{Theorem3.9} and Theorem~\ref{Theorem4.5}, the complex $_0\ZC_\bullet^+$ is a resolution of $R/J.$ Hence, the Hilbert function of $R/J$ can be written in terms of the Hilbert functions of the components of the complex $_0\ZC_\bullet^+$ which, according to the definition of $_0\ZC_\bullet^+,$  are just some direct sums of Koszul cycles of $I$ shifted by the twists appearing in the Koszul complex $K_\bullet(\gamma; S).$ Since the Hilbert functions of Koszul cycles are inductively calculated in terms of those of the Koszul homology modules, the Hilbert function of $R/J$ only depends on the Koszul homology modules of $I$ and on the degrees of the generators of $\ia.$
\end{proof}

Next, the important numerical invariant associated an algebraic or geometric object is the Castelnuovo-Mumford regularity. Assume that $R=\bigoplus_{n\geq 0}R_n$ is a positively graded Noetherian $^\ast$local ring of dimension $d$ over a Noetherian local ring $(R_0,\m_0).$ Set $\m=\m_0+R_+.$ Suppose that $I$ and $\ia$ are two homogeneous ideals of $R$ generated by homogeneous elements $f_1,\ldots,f_r$ and $a_1,\ldots,a_s,$ respectively. For a homogeneous ideal $\mathfrak{b},$ the sum of the degrees of a minimal generating set of $\mathfrak{b}$ is denoted by $\sigma(\mathfrak{b}).$ For a finitely generated graded $R$-module $M,$ the Castelnuovo-Mumford regularity of $M$ is defined as $\reg(M):=\max\{\ennd(H_{R_+}^i(M)) +i\}.$ In \cite{Ha12}, Hassanzadeh defined the regularity with respect to the maximal ideal $\m$ as $\reg_\m(M):=\max\{\ennd(H_\m^i(M))+i\}.$ He proved that
\begin{displaymath}
\reg(M)\leq \reg_\m(M)\leq \reg(M)+\dim(R_0),
\end{displaymath}
for any a finitely generated graded $R$-module $M,$ whenever $R$ is  a Cohen-Macaulay $^\ast\!$local ring, see \cite[Proposition~3.4]{Ha12}.

\smallskip
The next proposition improves \cite[Theorem~3.6]{Ha12} by removing
the arithmetic hypothesis of residual intersections.

\begin{pro} \label{Proposition5.2} 
Let $(R,\m)$ be a positively graded Cohen-Macaulay $^\ast\!$local ring over a Noetherian local ring $(R_0,\m_0)$ and $\ia\subset I$ be two homogeneous ideals, with $\hei(I)=g.$ Suppose that $I$ satisfies $\SD_1.$ Then, for any $s$-residual intersection $J=(\ia\colon_R I),$
\begin{align*}
\reg(R/J)\leq \reg(R)+\dim(R_0)+\sigma(\ia)-(s-g+1)\indeg(I/\ia) -s.
\end{align*}
\end{pro}
\begin{proof} 
The proof of this result goes along the same lines as in \cite[Theorem~3.6]{Ha12}. Indeed, Theorem~\ref{Theorem4.5} implies that $R/J$ is Cohen-Macaulay and is resolved by $_0\ZC_\bullet^+.$
\end{proof}

The next proposition improves the result of Hassanzadeh and the second named author in \cite[Proposition~3.3]{HN16}.

\begin{pro} \label{Proposition5.3} 
Let $(R,\m)$ be a positively graded Cohen-Macaulay $^\ast\!$local ring over a Noetherian local ring $(R_0,\m_0),$ with canonical module $\omega.$ Let $\ia\subset I$ be two homogeneous ideals, with $\hei(I)=g,$ and let $J=(\ia\colon_R I)$ be an $s$-residual intersection of $I.$  Suppose that $I$ satisfies $\SD_2$ and $\Tor_1^R(R/I,\omega)=0.$ Then
$$\omega_{R/J}=\Sym_R^{s-g+1}(I/\ia) \otimes_R \omega (\sigma(\ia)).$$
\end{pro}
\begin{proof}
The proof proceeds along the same lines as in the local case.  
\end{proof}
The following result is already an improvement of \cite[Proposition~3.15]{Ha12} and also of \cite[Proposition~3.3]{HN16}. We show the equality of the proposed upper bound for Castelnuovo-Mumford regularity of residual intersections in  Proposition~\ref{Proposition5.2}. This equality is showed  by Hassanzadeh  for perfect ideals of height 2 \cite[Theorem~3.16(iii)]{Ha12}.

\begin{cor}\label{Corollary5.4}
Under the assumptions of Proposition~\ref{Proposition5.3}, 
$$\reg_\m(R/J)=\reg_\m(R)+\sigma(\ia)-(s-g+1)\indeg(I/\ia)-s.$$
In particular, if $\dim(R_0)=0$ then 
$$\reg(R/J)=\reg(R)+\sigma(\ia)-(s-g+1)\indeg(I/\ia)-s.$$
\end{cor}
\begin{proof}
By Theorem~\ref{Theorem4.5}, $R/J$ is Cohen-Macaulay of dimension $d-s.$ By using the local duality theorem and Proposition~\ref{Proposition5.3}, 
\begin{align*}
\reg_\m(R/J)&=\ennd(H_\m^{d-s}(R/J))+d-s=-\indeg(\omega_{R/J})+d-s\\
&=\sigma(\ia)-\indeg(\Sym_R^{s-g+1}(I/\ia)\otimes_R \omega)+d-s\\
&=\sigma(\ia)-\indeg(\Sym_R^{s-g+1}(I/\ia))-\indeg( \omega)+d-s\\
&=\reg_\m(R)+\sigma(\ia)-\indeg(\Sym_R^{s-g+1}(I/\ia))-s,
\end{align*}
since  $\reg_\m(R)=\ennd(H_\m^d(R))+d=-\indeg(\omega)+d.$

\smallskip
It remains to prove that $\indeg(\Sym_R^{s-g+1}(I/\ia))= (s-g+1)\indeg(I/\ia).$ 

\smallskip
Let $g_1,\ldots, g_\ell$ be a minimal set of generators of $I/\ia.$ We have $$\Sym_R(I/\ia)\otimes_R R/\m \simeq  (R/\m) [Y_1,\ldots, Y_\ell],$$
where $Y_i$ is the class of $g_i$ in $\Sym_R(I/\ia) \otimes_R R/\m.$

\smallskip
Suppose that $\deg(g_1)=\indeg(I/\ia).$ Since 
$$\Sym_R(I/\ia) \otimes_R R/\m \simeq (R/\m)[Y_1,\ldots, Y_\ell]$$
is a polynomial ring, we see that $Y_1^{ s-g+1}\neq  0,$ hence $g_1^{s-g+1}\neq 0$ (this product in $\Sym_R(I/\ia))$ and $g_1^{s-g+1}\in  \Sym_R^{s-g+1}(I/\ia).$ Thus 
\begin{align*}
\indeg(\Sym_R^{s-g+1}(I/\ia))\leq \deg(g_1^{s-g+1})=(s-g+1)\deg(g_1)=(s-g+1)\indeg(I/\ia).
\end{align*}

\smallskip
On the other hand, $\indeg(\Sym_R^{s-g+1}(I/\ia))\geq (s-g+1)\indeg(I/\ia).$  Thus
\begin{align*}
\indeg(\Sym_R^{s-g+1}(I/\ia))= (s-g+1)\indeg(I/\ia).
\end{align*}

\smallskip
The remaining part follows from  $\reg(M)\leq \reg_m(M)\leq \reg(M)+\dim(R_0)$ for any finitely generated graded $R$-module $M.$
\end{proof}

Finally, we close this section by giving some tight relations between the Hilbert series of a residual intersection and the $(s-g+1)$-th symmetric power of $I/a.$ 
\begin{cor} \label{Corollary5.5}
Let $(R,\m)$ be a positively graded Cohen-Macaulay $^\ast$local algebra of dimension $d$ over an Artinian local ring $R_0,$ with canonical module $\omega.$ Suppose that $\ia \subset I$ are two homogeneous ideals of $R,$ with $\hei(I)=g,$ and  $J=(\ia\colon_R I)$ is an $s$-residual intersection of $I.$ Write 
\begin{align*}
H_{R/J}(t)=\frac{P(t)}{ (1-t^{a})^{d-s}},\quad H_{\Sym_R^{s-g+1}(I/\ia)\otimes_R \omega}(t)=\frac{Q(t)}{ (1-t^{a})^{d-s}},
\end{align*}
with $a$ the least common multiple of the degrees of the generators of the algebra $R$ over $R_0$ and $P(t),Q(t)\in \Z[t,t^{-1}],$ with $P(1),Q(1)>0.$ 
If $I$ satisfies $\SD_2$ and $\Tor_1^R(R/I,\omega)=0,$ then 
$$
P(t)=t^{\sigma(\ia)+a(d-s)} Q(t^{-1}).
$$
In particular, if $R$ is generated over $R_0$ by elements of degree one, that is, $R=R_0[R_1],$ then 
$$
e(R/J)=e(\Sym_R^{s-g+1}(I/\ia)\otimes_R \omega).
$$
\end{cor}  
\begin{proof}
By Proposition~\ref{Proposition5.3},
$$\omega_{R/J}\simeq \Sym_R^{s-g+1}(I/\ia)\otimes_R \omega (\sigma(\ia)).$$ 
It follows from \cite[Corollary~4.4.6]{BH98} that 
\begin{displaymath}
H_{\Sym_R^{s-g+1}(I/\ia)\otimes_R \omega(\sigma(\ia))}(t) =(-1)^{d-s}H_{R/J}(t^{-1})
\end{displaymath}
is equivalent to
\begin{equation*}
H_{\Sym_R^{s-g+1}(I/\ia)\otimes_R \omega}(t)=(-1)^{d-s}t^{\sigma(\ia)}H_{R/J}(t^{-1}).
\end{equation*}
Thus
\begin{displaymath}
Q(t)=t^{\sigma(\ia)+a(d-s)}P(t^{-1}) 
\end{displaymath}
gives 
$$P(t)=t^{\sigma(\ia)+a(d-s)}Q(t^{-1}).$$
In particular, 
$$
e(R/J)=P(1)=Q(1)=e(\Sym_R^{s-g+1}(I/\ia)\otimes_R \omega),
$$
by \cite[Proposition~4.1.9]{BH98}.
\end{proof}


\section{Duality for residual intersections of strongly Cohen-Macaulay ideals} \label{Section6}
The duality for residual intersetcions is a center of interest in during the development of the theory of residual. The first results of duality were proven by Peskine and Szpiro for the  theory of liaison in \cite{PS74}. Afterwards, around the works of Huneke and Ulrich in \cite{HU88}, generalizing the corresponding statement in the theory of linkage of Peskine and Szpiro. In particular, the recent works of Eisenbud and Ulrich in \cite{EU16} give some results on the duality for residual intersections. 

\smallskip
In this section, we provide the duality for residual intersections in the case  where $I$ is a strongly Cohen-Macaulay ideal. In this case, the structure of the canonical module of some symmetric powers of $I/\ia$ is given.  Therefore, we may establish some tight relations between the Hilbert series of the symmetric powers of $I/\ia$ and we give the closed formulas for the type and for the Bass number of  $\Sym_R^k(I/\ia).$

\smallskip
First we prove on the duality of residual approximation complexes in the height two case.
\begin{pro} \label{Proposition6.1}
Let $(R,\m)$ be a Cohen-Macaulay  local ring of dimension $d,$ with canonical module $\omega,$ and let $\ia\subset I$ be two ideals of $R.$ Suppose that $I$ is a  strongly Cohen-Macaulay ideal of height 2 and $J=(\ia\colon_R I)$ is an $s$-residual intersection of $I.$  Then, for all $0\leq k\leq s-2,$ 
\begin{align*}
\omega_{H_0(_k\ZC_\bullet^+)}\simeq H_0(_{s-k-1}^{\qquad\omega}\!\ZC_\bullet^+).
\end{align*}
\end{pro}
\begin{proof} 
By Theorem~\ref{Theorem3.9},  the complex $_k\ZC_\bullet^+$ is acyclic and $H_0(_k\ZC_\bullet^+)$ is Cohen-Macaulay of dimension $d-s.$ Therefore, by local duality,
\begin{equation} \label{Equation6.1}
\omega_{H_0(_k\ZC_\bullet^+)}^\vee\simeq H_\m^{d-s}(H_0(_k\ZC_\bullet^+)).
\end{equation}
	
\smallskip
 As $I$ is strongly Cohen-Macaulay of height 2, we have that $\depth(Z_i)=d,$ for all $0\leq i\leq r-1.$ By the definition of $_k\ZC_\bullet^+,$ for all $0\leq i\leq s,\ _k\ZC_i^+$ is a direct sum of copies of modules $Z_0,Z_1,\ldots,Z_{r-1},$ therefore $\depth(_k\ZC_i^+)=d.$ We now consider the double complex $C_\m^\bullet (_k\ZC_\bullet^+)$ that gives rise to two sequences. The second terms of the horizonal spectral are
	\begin{align*}
	^2\textbf{E}_{\hor}^{-i,-j}=
	\begin{cases}
	H_\m^{d-s}(H_0(_k\ZC_\bullet^+))&\text{if} \quad j=d-s\; \text{and}\  i=0\\0 & \text{otherwise}
	\end{cases}
	\end{align*}
and the first terms of the vertical spectral are
\begin{footnotesize}
\begin{align*}
^1\textbf{E}_{\ver}^{-i,-j}=
\begin{cases}
\xymatrix{0\ar[r]& H_\m^d(_k\ZC_s^+)\ar[r]&\cdots \ar[r] & H_\m^d(_k\ZC_1^+)\ar[r]&H_\m^d(_k\ZC_0^+)\ar[r]& 0}\; \text{if} \; j=d\\
0 \qquad\quad \text{otherwise}.
\end{cases}
\end{align*}
\end{footnotesize}

\smallskip
By the convergence of the spectral sequences, we obtain
\begin{equation} \label{Equation6.2}
H_\m^{d-s}(H_0(_k\ZC_\bullet^+))\simeq \, ^\infty\textbf{E}_{\ver}^{-s,-d}= \, ^2\textbf{E}_{\ver}^{-s,-d}.
\end{equation}
By  Lemma~\ref{Lemma4.3}(iv),  we have a commutative diagram, for all $0\leq k\leq s-2$
\begin{displaymath}
\xymatrix{  H_\m^d(_k\ZC_s^+) \ar[r]\ar[d]_{\simeq} & H_\m^d(_k\ZC_{s-1}^+) \ar[d]^{\simeq } \\ 
(_{s-k-1}^{\qquad \omega}\ZC_0^+)^\vee\ar[r] &  (_{s-k-1}^{\qquad \omega}\ZC_1^+)^\vee.}
	\end{displaymath} 
	Therefore
	\begin{equation} \label{Equation6.3}
	^2\textbf{E}_{\ver}^{-s,-d} \simeq H_0( _{s-k-1}^{\qquad\omega}\!\ZC_\bullet^+)^\vee.
	\end{equation}
	
	\smallskip
	By (\ref{Equation6.1}), (\ref{Equation6.2}), (\ref{Equation6.3}) and Lemma~\ref{Lemma2.3}, 
	$$\omega_{H_0(_k\ZC_\bullet^+)}\simeq H_0( _{s-k-1}^{\qquad\omega}\!\ZC_\bullet^+).$$
\end{proof}

We now state the main result of this section. Let recall us that if $M,N,L$ are three $R$-modules, then  a morphism $\varphi:~M \otimes_R N \longrightarrow L$ is a perfect pairing if $\psi_1: M \longrightarrow \Hom_R(N,L),$ sending $m$ to $\psi_1(m): n\mapsto \varphi(m\otimes n)$ and $\psi_2: N \longrightarrow \Hom_R(M,L),$ sending $n$ to $\psi_2(n): m\mapsto \varphi(m\otimes n)$ are two isomorphisms. 
\begin{theo} \label{Theorem6.2}
Let $(R,\m)$ be a Cohen-Macaulay  local ring of dimension $d,$ with canonical module $\omega,$ and let  $\ia\subset I$ be two ideals of $R,$ with $\hei(I)=g.$ Suppose that $J=(\ia\colon_R I)$ is an $s$-residual intersection of $I.$ If $I$ is strongly Cohen-Macaulay and  $\Tor_1^R(R/I,\omega)=0,$ then, for all $1\leq k\leq s-g,$
\begin{enumerate}
\item[\em (i)]  the canonical module of $\Sym_R^k(I/\ia)$ is 
$\Sym_R^{s-g+1-k}(I/\ia)\otimes_R \omega;$
\item[\em (ii)] there is a perfect pairing
\begin{displaymath}
(\Sym_R^k(I/\ia)\otimes_R \omega)\otimes_R \Sym_R^{s-g+1-k}(I/\ia)\longrightarrow \Sym_R^{s-g+1}(I/\ia)\otimes_R \omega.
\end{displaymath}
\end{enumerate}
\end{theo}

\begin{proof} 
(i) First we treat the case $g = 2$. By Proposition~\ref{Proposition6.1}, for all $1\leq k\leq s-2,$ 
\begin{align*} 
\omega_{\Sym_R^k(I/\ia)}
&\simeq  H_0(_{s-k-1}^{\qquad\omega}\!\ZC_\bullet^+)\\
&\simeq \Sym_R^{s-k-1}(I/\ia)\otimes_R \omega.
\end{align*}
The last isomorphism follows from Proposition~\ref{Proposition3.3}.

\smallskip
Now, we may suppose that $g\geq 2$ by Remark~\ref{Remark3.8}. If $g>2,$ then we choose a regular sequence $\aa$ of length $g-2$ inside $\ia$ which is a part of a minimal generating set of $\ia$ as in the proof of Theorem~\ref{Theorem4.8}. As $I/\aa$ is strongly Cohen-Macaulay by Proposition~\ref{Proposition4.1}, it follows from the height two case that
\begin{align*}
\omega_{\Sym_R^k(I/\ia)}\simeq \omega_{\Sym_{R/\aa}^k(\frac{I/\aa}{\ia/\aa})}&\simeq \Sym_{R/\aa}^{(s-g+2)-k-1}\Big(\frac{I/\aa}{\ia/\aa}\Big)\otimes_{R/\aa} (\omega/\aa\omega)\\
&\simeq \Sym_{R}^{s-g+1-k}(I/\ia)\otimes_R\omega.
\end{align*}

(ii) It suffices to prove  that, for all $1\leq k\leq s-g,$
\begin{align*}
\Sym_R^k(I/\ia)\otimes_R \omega\simeq  \Hom_R\big(\Sym_R^{s-g+1-k}(I/\ia), \Sym_R^{s-g+1}(I/\ia)\otimes_R \omega\big).
\end{align*}

As  $\Sym_R^{s-g+1-k}(I/\ia)$ is a maximal Cohen-Macaulay $R/J$-module by Corollary~\ref{Corollary4.6}(ii) and $\Sym_R^{s-g+1}(I/\ia)\otimes_R \omega$ is the canonical module of $R/J$ by Theorem~\ref{Theorem4.8},
$$\omega_{\Sym_R^{s-g+1-k}(I/\ia)}\simeq \Hom_R\big(\Sym_R^{s-g+1-k}(I/\ia), \Sym_R^{s-g+1}(I/\ia)\otimes_R \omega\big).$$

The conclusion follows from (i).
\end{proof}
In particular, if the residual intersections are geometric, we obtain the following results that could be compared to one of \cite[Theorem~2.2]{EU16}.
\begin{cor} \label{Corollary6.3}
Let $(R,\m)$ be a Gorenstein  local ring of dimension $d$ and let $\ia\subset I$ be two ideals of $R.$  Assume that $I$ is a  strongly Cohen-Macaulay ideal of height $g$ and $ J=(\ia\colon_R I)$ is a geometric  $s$-residual intersection of $I.$ Then, for all $1\leq k\leq s-g,$
\begin{enumerate}
\item[\em (i)]  the canonical module of $I^k/\ia I^{k-1}$ is 
$I^{s-g+1-k}/\ia I^{s-g-k};$
\item[\em (ii)] there is a perfect pairing
\begin{displaymath}
I^k/\ia I^{k-1}\otimes_R I^{s-g+1-k}/\ia I^{s-g-k}\longrightarrow I^{s-g+1}/\ia I^{s-g}.
\end{displaymath}
\end{enumerate}
\end{cor}
\begin{proof}
It is an immediate translation from Theorem~\ref{Theorem6.2}, in view of the facts that $\Sym_R^k(I/\ia)\simeq I^k/\ia I^{k-1}$ by \cite[Corollary~2.11]{HN16} and $\omega_R\simeq R.$
\end{proof} 

Notice that the pairing in this Corollary, and in the main Theorem above need not be given by multiplication. However, Eisenbud and Ulrich proved that, in many situations where our results apply, the multiplication indeed produces a perfect pairing. In this regards, an example they provide is interesting.

\begin{exem}\cite[Example~2.8]{EU16}
Let $R=k[[x,y,z]],$ where $k$ is an infinite field and $I=(x,y)^2.$ If $\ia$ is generated by 3 sufficiently general elements of degree 3 in $I,$ then $J=\ia\colon_R I$ is a 3-residual intersection. Using  {\tt Macaulay2} \cite{Macaulay2}, they verified that $I$ is strongly Cohen-Macaulay, hence $\omega_{R/J}\simeq \Sym_R^2(I/\ia).$ Moreover  $\omega_{I/\ia}\simeq I/\ia.$ 

\smallskip
Computation shows that there is a unique (up to scalars) nonzero homogeneous map $I/\ia\otimes_R I/\ia \longrightarrow \omega_{R/J}$ of lowest degree, and this is a perfect pairing. But they notice that there can be no perfect pairing $I/\ia\otimes_R I/\ia \longrightarrow I^2/\ia I$ because the target is annihilated by $(x,y,z)^2$ while $I/\ia$ is not. This implies that $\omega_{R/J}\neq I^2/\ia I$ and $J$ is not geometric.

However, the multiplication with value in the symmetric square $I/\ia\otimes_R I/\ia \longrightarrow \Sym_R^2(I/\ia)$ is a perfect pairing.
\end{exem}

Next, we will show that the perfect paring in Theorem~\ref{Theorem6.2} and also in Corollary~\ref{Corollary6.3} could be chosen by multiplication. Fisrt, we need the following lemmas. 

\begin{lem} \label{Lemma6.5}
Let $(R,\m ,k)$ be a local Noetherian ring and $S$ be a Noetherian standard graded $R$-algebra.  For any $s\geq t,$ we consider 
$$
\psi : S_t \longrightarrow \Hom_R(S_{s-t},S_s)
$$ 
the natural map given by the algebra structure of $S.$ If $H_{S_{+}}^0(S\otimes_R k)_t=0,$ then  $\psi \otimes k$ is into.
\end{lem}
\begin{proof}
Let $L\in S_t$ be such that $0\neq \overline{L}\in S_t\otimes_R k=(S\otimes_R k)_t$. The element $L$ is sent to the class of the  homomorphism $\times L$. We have to prove that this class is not zero. 
As
$$
\m \Hom_R(S_{s-t},S_s)\subseteq \Hom_R(S_{s-t},\m S_s),
$$
it suffices to show that the image of $\times L$ is not contained in $\m S_s .$ The assertion is obvious if $s=t.$ If $s>t$, as $\overline{L}\notin H_{S_{+}}^0(S\otimes_R k)_t$ and $S_{s-t}=(S_{+})^{s-t}$, there exist $u\in S_{s-t}$ such that $\overline{L}.\overline{u}\not= 0$. Hence the image of $\times L$ contains $L.u\not\in \m S_s.$
\end{proof}

\begin{lem} \label{Lemma6.6}
Let $(R,\m ,k)$ be a local Noetherian ring and $M$ be a finitely generated $R$-module. For any $s\geq t,$ if there exists a $R$-module isomorphism 
$$\varphi : \Hom_R(\Sym^{s -t}_R(M),\Sym^{s}_R(M))\longrightarrow \Sym^{t}_R(M),$$ 
then the natural map given by the algebra structure of $\Sym_R(M)$
$$
\psi : \Sym^{t}_R(M) \longrightarrow \Hom_R(\Sym^{s -t}_R(M),\Sym^{s}_R(M))
$$ 
 is an isomorphism.
\end{lem}

\begin{proof}
The assertion of the lemma is  equivalent to show that $\varphi \circ \psi$ is onto, which in turn is equivalent to $\psi \otimes_R k$ being into (or equivalently onto).\\

Choose $\tau_1 : R^n \longrightarrow M$ onto with $n$ minimal (equivalently such that $R^n\otimes_R k\simeq M\otimes_R k$ via $\tau_1$). 
Then $\tau :=\Sym_R(\tau_1 ):\Sym_R(R^n)\longrightarrow \Sym_R(M)$ is onto and  $\tau\otimes_R k$ is an isomorphism identifying $\Sym_R(M)\otimes_R k$ with a polynomial ring in $n$ variables. It follows that $S=\Sym_R(M)$ satisfies the condition of Lemma~\ref{Lemma6.5}, hence $\psi \otimes k$ is into.
\end{proof}
Note that  $\Sym_{R/J}^0(I/\ia)=R/J$ and $\Sym_{R/J}^k(I/\ia)=\Sym_{R}^k(I/\ia)$ for $k>0$. We have the following results.
\begin{theo}\label{Theorem6.7}
Let $(R,\m)$ be a Gorenstein local ring  and let  $\ia\subset I$ be two ideals of $R,$ with $\hei(I)=g.$ Suppose that $J=(\ia\colon_R I)$ is an $s$-residual intersection of $I.$ If $I$ is strongly Cohen-Macaulay, then $\omega_{R/J}\simeq \Sym_{R/J}^{s-g+1}(I/\ia)$ and for all $0\leq k\leq s-g+1$
\begin{enumerate}
\item[\em (i)] the $R/J$-module $\Sym_{R/J}^{k}(I/\ia)$ is faithful and Cohen-Macaulay,
\item[\em (ii)] the multiplication
\begin{displaymath}
\Sym_{R/J}^k(I/\ia)\otimes_{R/J} \Sym_{R/J}^{s-g+1-k}(I/\ia)\longrightarrow \Sym_{R/J}^{s-g+1}(I/\ia)
\end{displaymath}
is a perfect pairing, 
\item[\em (iii)] setting $A:=\Sym_{R/J}(I/\ia)$, the graded $R/J$-algebra
$$\overline{A}:=A/A_{>s-g+1}=\bigoplus_{i=0}^{s-g+1}\Sym_{R/J}^i(I/\ia) 
$$
is Gorenstein.
\end{enumerate}
\end{theo}

\begin{proof}
The first item is Corollary \ref{Corollary4.6} (ii). The second and last items directly  follow from Lemma~\ref{Lemma6.6} together with Theorem~\ref{Theorem6.2} (ii) and (i), respectively. 
\end{proof}


\begin{cor} \label{Corollary6.5}
Let $(R,\m)$ be a positively graded Cohen-Macaulay $^\ast$local algebra of dimension $d$ over an Artinian local ring $R_0,$ with canonical module $\omega.$ Suppose that $\ia \subset I$ are two homogeneous ideals of $R,$ with $\hei(I)=g,$ and  $J=(\ia\colon_R I)$ is an $s$-residual intersection of $I$. Write 
\begin{align*}
H_{\Sym_R^k(I/\ia)}(t)=\frac{P_k(t)}{ (1-t^{a})^{d-s}},\quad H_{\Sym_R^k(I/\ia)\otimes_R \omega}(t)=\frac{Q_k(t)}{ (1-t^{a})^{d-s}},
\end{align*}
with $a$ the least common multiple of the degrees of the generators of the algebra $R$ over $R_0$ and $P_k(t),Q_k(t)\in \Z[t,t^{-1}],$ with $P_k(1),Q_k(1)>0,$ for each $1\leq k\leq s-g$.
If $I$ is strongly Cohen-Macaulay and $\Tor_1^R(R/I,\omega)=0,$ then 
$$
P_k(t)=t^{\sigma(\ia) +a(d-s)}Q_{s-g+1-k}(t^{-1}).
$$
In particular, if $R$ is generated over $R_0$ by elements of degree one, that is, $R=R_0[R_1],$ then 
$$
e(\Sym_R^k(I/\ia))=e(\Sym_R^{s-g+1-k}(I/\ia)\otimes_R \omega).
$$
\end{cor}  
\begin{proof}
The proof is analogous to one of Corollary~\ref{Corollary5.5}. It follows from the fact that   
\begin{equation*}
H_{\Sym_R^k(I/\ia)\otimes_R \omega}(t)=(-1)^{d-s}t^{\sigma(\ia)}H_{\Sym_R^{s-g+1-k}(I/\ia)}(t^{-1}).
\end{equation*}
\end{proof}
The next corollary enables us to calculate the type of some symmetric powers of $I/\ia.$ This is comparable with the results of Hassanzadeh and the second named author in \cite[Theorem~2.12]{HN16}.
\begin{cor} \label{Corollary6.6}
Let $(R,\m)$ be a Cohen-Macaulay local ring of dimension $d,$ with canonical module $\omega,$ and let $\ia\subset I$ be two ideals of $R,$ with $\hei(I)=g.$ Suppose that $J=(\ia\colon_R I)$ is an $s$-residual intersection of $I.$ If  $I$ is strongly Cohen-Macaulay and $\Tor_1^R(R/I,\omega)=0,$ then, for each $1\leq k\leq s-g,$
\begin{align*}
r(\Sym_R^{k}(I/\ia))=\binom{\mu(I/\ia)+s-g-k}{\mu(I/\ia)-1}r(R).
\end{align*}
\end{cor}

\begin{proof}
The proof is totally similar to one of Corollary~\ref{Corollary4.11}.
For all $1\leq k\leq s-g,$ 
\begin{align*}
r(\Sym_R^{k}(I/\ia))&=\mu(\Sym_R^{s-g+1-k}(I/\ia)\otimes_R \omega),
\end{align*}
by Theorem~\ref{Theorem6.2}(i) and \cite[Proposition~3.3.11]{BH98}.
\end{proof}

Let $R$ be a Noetherian ring, $M$ be a finitely generated $R$-module and $\point\in \Spec(R).$ The finite number  $$\mu_i(\point,M)=\dim_{k(\point)}(\Ext_{R_\point}^{i}(k(\point),M_\point))=\dim_{k(\point)}(\Ext_R^{i}(R/\point,M)_\point)$$
is called the \textit{$i$-th Bass number} of $M$ with respect to $\point,$ where $k(\point)=R_\point/\point R_\point.$ If $R$ is local, then $r(M)=\mu_{\depth(M)}(\m, M).$ These numbers have an interpretation in terms of the minimal injective resolution of $M,$ (see \cite[Proposition~3.2.9]{BH98}). The next corollary enables us to calculate the Bass numbers of some symmetric powers of $I/\ia.$

\begin{cor} \label{Corollary6.7}
Under the assumptions of Corollary~\ref{Corollary6.6}.  Let $\point$ be a prime ideal containing $J$ of $R,$ with  $\hei(\point)=i,$  then, for every $1\leq k\leq s-g,$
\begin{displaymath}
\mu_{i-s}(\point,\Sym_R^{k}(I/\ia)\otimes_R \omega)=\binom{\mu((I/\ia)_\point)+s-g-k}{\mu((I/\ia)_\point)-1}.
\end{displaymath}
 \end{cor}

\begin{proof}
By Theorem~\ref{Theorem4.5}, $R/J$ is Cohen-Macaulay of dimension $d-s$ and by Corollary~\ref{Corollary4.6}(ii),  $\Sym_R^{k}(I/\ia)$ is a maximal Cohen-Macaulay faithful $R/J$-module, for all $1\leq k\leq s-g+1.$ Furthermore, by Theorem~\ref{Theorem6.2}(i),  $\Sym_R^{k}(I/\ia)\otimes_R \omega$ is a  maximal Cohen-Macaulay faithful $R/J$-module, for all $1\leq k\leq s-g.$ 

\smallskip
Suppose that $J\subset\point_s\subsetneq \point_{s+1}\subsetneq\cdots \subsetneq \point_{i}=\point$ is a maximal chain of primes of $\Spec(R/J)$ contained in $\point.$ Let $b_j\in \point_j-\point_{j-1},$ for all $s+1\leq j\leq i.$ Then $\ib=(b_{s+1},\ldots,b_i)$  is a regular sequence over $R/J$ and therefore also over $\Sym_R^{k}(I/\ia)$ and $\Sym_R^{k}(I/\ia)\otimes_R \omega$, for all $1\leq k\leq s-g.$ 

\medskip
For $1\leq k\leq s-g,$ $(b_{s+1},\ldots,b_i)$ is a regular sequence over $(\Sym_R^{k}(I/\ia)\otimes_R \omega)_\point$ and annihilates $k(\point)$,  hence \cite[Lemma~1.2.4]{BH98} gives
\begin{align*}
&\Ext_{R_\point}^{i-s}(k(\point), (\Sym_R^{k}(I/\ia)\otimes_R \omega)_\point)\simeq \Hom_{R_\point}\big(k(\point),(\Sym_R^{k}(I/\ia)\otimes_R \omega)_\point\otimes_{R_\point} R_\point/\ib R_\point\big)\\
&\qquad \qquad \qquad \qquad \simeq \Hom_{R_\point}\Big(k(\point),\Hom_R(\Sym_R^{s-g+1-k}(I/\ia),\omega_{R/J})\otimes_R R_\point/\ib R_\point\Big).
\end{align*}
The last isomorphism follows from Theorem~\ref{Theorem4.8} and Theorem~\ref{Theorem6.2}(ii). By \cite[Proposition~3.3.3]{BH98}
\begin{align*}
\Hom_R\big(\Sym_R^{s-g+1-k}&(I/\ia), \omega_{R/J}\big)\otimes_R R_\point/\ib R_\point\\
&\simeq \Hom_{R_\point}\Big(\Sym_R^{s-g+1-k}(I/\ia)\otimes_R R_\point/\ib R_\point,\omega_{R/J}\otimes_R R_\point/\ib R_\point\Big).
\end{align*}
Thus, we obtain
\begin{small}
\begin{align*}
\Ext_{R_\point}^{i-s}(k(\point),&(\Sym_R^{k}(I/\ia)\otimes_R \omega)_\point)\\
&\simeq \Hom_{R_\point}\Big(k(\point),\Hom_{R_\point}\big(\Sym_R^{s-g+1-k}(I/\ia)\otimes_R R_\point/\ib R_\point,\omega_{R/J}\otimes_R R_\point/\ib R_\point\big)\Big)\\
&\simeq \Hom_{R_\point}\Big(k(\point)\otimes_{R_\point} \Sym_R^{s-g+1-k}(I/\ia)\otimes_R R_\point/\ib R_\point,\omega_{R/J}\otimes_R R_\point/\ib R_\point\Big)\\
&\simeq \Hom_{R_\point}\Big( k(\point)\otimes_R\Sym_R^{s-g+1-k}(I/\ia),\omega_{R/J}\otimes_R R_\point/\ib  R_\point\Big)\\
&\simeq \Hom_{R_\point}\Big( \Sym_{k(\point)}^{s-g+1-k}(k(\point)\otimes_R I/\ia),\omega_{R/J}\otimes_R R_\point/\ib R_\point\Big).
\end{align*}
\end{small}
Since  $k(\point)\otimes_R I/\ia\simeq k(\point)\otimes_{R_\point} (I/\ia)_\point$ is a $k(\point)$-vector space of dimension $\mu_\point :=\mu((I/\ia)_\point),$ 
 \begin{align*}
 \Sym_{k(\point)}(k(\point)\otimes_R I/\ia)\simeq k(\point)[Y_1,\ldots,Y_{\mu_\point}].
 \end{align*}
 It follows that
 \begin{align*}
 \Ext_{R_\point}^{i-s}(k(\point),(\Sym_R^{k}(I/\ia)\otimes_R\omega)_\point)&\simeq \Hom_{R_\point}\Big( k(\point)^{\binom{\mu_\point +s-g-k}{\mu_\point -1}},\omega_{R/J}\otimes_R R_\point/\ib R_\point\Big)\\
 &\simeq \Hom_{R_\point}\big( k(\point),\omega_{R/J}\otimes_R R_\point/\ib R_\point\big)^{\binom{\mu_\point +s-g-k}{\mu_\point -1}}\\
 &\simeq \Ext_{R_\point}^{i-s}\big( k(\point),(\omega_{R/J})_\point\big)^{\binom{\mu_\point +s-g-k}{\mu_\point -1}}.
 \end{align*}
The last isomorphism follows from the fact that  $\ib R_\point$ is regular over $(\omega_{R/J})_\point$ and annihilates $k(\point).$ Therefore
\begin{align*}
 \mu_{i-s}(\point,\Sym_R^{k}(I/\ia)\otimes_R\omega)& =
\binom{\mu_\point +s-g-k}{\mu_\point -1}\mu_{i-s}(\point,\omega_{R/J})\\
&=
\binom{\mu_\point +s-g-k}{\mu_\point -1},
\end{align*}
 since $\hei(\point)=i-s$ in $R/J,\ \mu_{i-s}(\point,\omega_{R/J})=1,$ by \cite[Theorem~3.3.10]{BH98}.
\end{proof}

\section*{Acknowledgments}
The authors would like to thank Seyed Hamid Hassanzadeh for useful comments on a first version that leaded to important improvements. A part of this work was done while the second named author was visiting the Universit\'e Pierre et Marie Curie  and he expresses his gratitude for this hospitality. All authors are partially supported by the Math-AmSud program SYRAM that gave them the opportunity to work together on this question.


\end{document}